\documentclass[final]{siamltex}

\usepackage[leqno]{amsmath}
\usepackage{amsfonts}
\usepackage{subfigure}
\usepackage{graphicx}
\usepackage{color}
\newcommand{\R}{\mathbb R}

\newtheorem{remark}[theorem]{Remark}

\title{Conservative methods for stochastic differential equations with a conserved quantity
\thanks{C.~Chen and J.~Hong are supported by National Natural Science Foundation of China (NO. 91130003, NO. 11021101 and NO. 11290142).
D.~Cohen is supported by UMIT Research Lab at Ume{\aa} University and the Swedish Research Council (VR).}}

\author{Chuchu Chen\thanks{Institute of Computational Mathematics and
Scientific/Engineering Computing, Chinese Academy of Sciences, Beijing, PR China. ({\tt chenchuchu@lsec.cc.ac.cn})({\tt hjl@lsec.cc.ac.cn}).}
        \and David Cohen\thanks{Matematik och matematisk statistik, Ume{\aa} universitet, 90187 Ume{\aa}, Sweden.({\tt david.cohen@math.umu.se})}
        \and Jialin Hong$^{\dag}$}

\begin{document}

\maketitle

\begin{abstract}
This paper proposes a novel conservative method for numerical computation of
general stochastic differential equations in the Stratonovich sense
with a conserved quantity.
We show that the mean-square order of the method is $1$ if
noises are commutative and that the weak order is also $1$.
Since the proposed method may need the computation of a
deterministic integral, we analyse the effect
of the use of quadrature formulas on the convergence orders.
Furthermore, based on the splitting technique of stochastic vector fields,
 we construct conservative
composition methods with similar orders as the above method.
Finally, numerical experiments are presented to support
 our theoretical results.
\end{abstract}

\begin{keywords}
stochastic differential equations, invariants,
conservative methods, quadrature formula, splitting technique,
mean-square convergence order, weak convergence order
\end{keywords}

\begin{AMS}
60H10, 60H35, 65C20, 65C30, 65D30
\end{AMS}

\pagestyle{myheadings}
\thispagestyle{plain}
\markboth{C.~Chen, D.~Cohen and J.~Hong}{Conservative methods for SDEs}

\section{Introduction}
\label{intro}
In this paper, we consider general $d$-dimensional autonomous stochastic differential equations (SDE)
in the Stratonovich sense
\begin{equation}\label{eq1}
 \mathrm d X(t)=f(X(t))\,\mathrm  dt+\sum_{r=1}^{m}g_{r}(X(t))\circ\,\mathrm  dW_{r}(t),
\quad 0\leq t\leq T, \quad X(0)=X_{0},
\end{equation}
where $W_{r}(t)$, $r=1,\cdots, m$ are $m$ independent one-dimensional Brownian motions,
defined on a complete probability space $(\Omega, \mathcal{F}, \{\mathcal{F}_{t}\}_{t\geq 0}, P)$.
The initial value $X_{0}$ is $\mathcal{F}_{t_{0}}$-measurable with $E|X_{0}|^{2}<\infty$.
Here, $f\colon\R^d\to\R^d$ and $g_{r}\colon\R^d\to\R^d$ are such that the above problem possesses a unique
solution.
The studies of SDE \eqref{eq1} have drawn dramatic attentions due to its applications in physics, engineering, economics, etc., concerning
the effects of random-phenomena.
 Furthermore, we will assume that equation \eqref{eq1} possesses a scalar conserved
quantity $I(x)$, which means that $\text dI(X(t))=0$ along the exact solution $X(t)$ of \eqref{eq1},
e.~g. see \cite{b81,Cohen2013,Faou2009,Zhang2011,Misawa1999} and references therein for the applications and studies of conservative SDE. Our aim is to derive and
analyse numerical methods for \eqref{eq1} preserving this conserved quantity.

Finding numerical solutions of stochastic differential equations is an active ongoing research area,
see the review paper \cite{Burrage2004}, the monographs \cite{Kloeden1992,Milstein1995}
and references therein for instance. Further, it is important to design numerical schemes
which preserve the properties of the original problems as much as possible.
References \cite{Abdulle2012,c12,Malham2008,Milstein2002a,Milstein2002,Moro2007,Schurz1999b,Schurz1999a,Wang},
without being exhaustive, show general improvements
of these so-called geometric numerical methods over more traditional numerical schemes
such as Euler-Maruyama's method or the Milstein scheme.

Concerning our problem \eqref{eq1} with a conserved quantity, \cite{Misawa1999} develops
a method to derive conserved quantities from symmetry of SDEs in Stratonovich sense.
Further, \cite{Misawa2000} proposes an energy-preserving method for stochastic Hamiltonian dynamical systems
and presents the local error order of the method. The recent work \cite{Cohen2013} proposes
a new energy-preserving scheme for stochastic Poisson systems with non-canonical structure matrix
and shows that the mean-square convergence order of the scheme is $1$.
For general SDEs driven by one-dimensional Brownian motion in Stratonovich sense,
the authors of \cite{Zhang2011} propose two conservative methods by means of
the skew gradient form of the original SDEs (see below for more details).
They also prove that these two methods are convergent with accuracy $1$
in the mean-square sense.
Based on these two last references, we propose new conservative numerical methods
for general stochastic differential equations with a conserved quantity in the present paper.


Since the problem of computing expectations of functionals
of solutions to SDEs appears in many applications \cite{talay84}, for example:
in finance \cite{RAT04}, in random mechanics \cite{SPA99},
or in bio-chemistry \cite{Gardiner:1985:HSM};
we will not only derive the mean-square, but also weak
convergence orders of new invariant-preserving numerical methods.
Comparing our scheme with the Milstein method, we prove that the mean-square
convergence order of our method is $1$ under the condition of commutative noise.
Furthermore, without assuming any commutativity condition,
we show that the weak convergence order of our method is $1$.
Since the proposed method may need the computation of a deterministic integral,
we will also analyse the effect of the use of quadrature formulas on convergence orders.
We will show that if the order of a quadrature formula is greater than $2$,
the mean-square and weak orders of our method remains $1$.
Based on the splitting technique of stochastic vector fields, we derive new invariant-preserving composition methods of mean-square
order one (in the commutative case) and weak order one.

This paper is organized as follows. Section~\ref{sec-meth}
presents the skew gradient form of the problem and derive the proposed invariant-preserving scheme.
Properties of the numerical scheme are analyzed in Section~\ref{sec-prop}.
The effects of quadrature formula on the mean-square and weak convergence orders
and on the discrete conserved quantity are investigated in Section~\ref{sec-qf}.
Section~\ref{sec-split} deals with the splitting technique of stochastic vector field.
Finally, numerical examples are presented to support
the theoretical analysis of the previous sections
in Section~\ref{sec-num}.

In the sequel, we will make use of the following notations.
\begin{itemize}
\item $|x|$ is the Euclidean norm of a vector $x$ or the induced norm for a matrix.
\item We use superscript indices to denote components of a vector or a matrix.
\item Partial derivatives are denoted $\partial_{i}:=\dfrac{\partial}{\partial x^{i}}$
and $\partial_{ij}:=\dfrac{\partial^{2}}{\partial x^{i}\partial x^{j}}$ etc.
\item $C^{k}_{b}(\R^{d_{1}},\R^{d_{2}})$ is the space of $k$ times continuously differentiable functions
$g\colon\R^{d_{1}}\to\R^{d_{2}}$ with uniformly bounded derivatives (up to order $\leq k$).
\item $C_{P}^{k}(\R^{d},\R)$ denotes the space of all $k$ times continuously
differentiable functions $f\colon\R^d\to\R$ with polynomial growth,
i.\,e., there exists a constant $C>0$ and $r\in \mathbb{N}$,
such that $|\partial^{j}f(x)|\leq C(1+|x|^{2r})$
for all $x\in\R^{d}$ and any partial derivative of order $j\leq k$.
\end{itemize}

\section{Presentation of the conservative method for skew gradients problems}\label{sec-meth}
In this section, we will first present the equivalent skew gradient form of \eqref{eq1}
with a conserved quantity $I$, and then we will define our invariant-preserving numerical
method.

The equivalent skew gradient form of \eqref{eq1} is stated below.
\begin{proposition}[See Theorem~2.2 in \cite{Zhang2011} for a one-dimensional Brownian motion]
The $d$-dimensional system \eqref{eq1} with a scalar conserved quantity $I(x)$ is
equivalent to the following skew gradient (SG) form
\begin{equation}\label{SG system}
\text dX(t)=S(X)\nabla I(X)\text dt+\sum_{r=1}^{m}T_{r}(X)\nabla I(X)\circ\text dW_r(t),
\end{equation}
where $S(X)$, $T_{r}(X)\in\R^{d\times d}$ are skew symmetric matrices such
that $S(X)\nabla I(X)=f(X)$ and $T_{r}(X)\nabla I(X)=g_{r}(X)$ for $r=1,\cdots,m$.
\end{proposition}

Note that the proof of the above proposition is similar to the one
of Theorem~2.2 in \cite{Zhang2011}.
Further it makes use of constructive techniques.
It not only proves the validity of the proposition, but also presents the construction
of the skew symmetric matrices $S(X)$ and $T_{r}(X)$. For example, one can take
\[
S(x)=\dfrac{f(x)a(x)^{T}-a(x)f(x)^{T}}{a(x)^{T}\nabla I(x)},\quad T_{r}(x)=
\dfrac{g_{r}(x)b(x)^{T}-b(x)g_{r}(x)^{T}}{b(x)^{T}\nabla I(x)},
\]
where $A^T$ denotes the transpose of $A$. Here $a(x)$, $b(x)$
are arbitrary column vectors such that
$a(x)^{T}\nabla I(x)\neq 0$, $b(x)^{T}\nabla I(x)\neq 0$.

\begin{remark}
Since we will make use of general theorems (\cite[Theorem~2.1, Sect. 2.2.1]{Milstein1995} or
\cite[Theorem~14.5.2]{Kloeden1992} for instance) to prove convergence of our numerical method,
we will assume that $I$, $S$ and $T_{r}$ $(r=1,\cdots,m)$
are smooth functions with globally bounded derivatives up to certain order.
Observe however that, in certain cases, we may get rid off these restrictions
thanks to the invariant preservation property of the numerical scheme \eqref{method Cohen}
(see \cite[Remarks~3.4,~3.5 and Theorem~3.4]{Cohen2013} for instance).
\end{remark}

We now present the conservative numerical method for \eqref{eq1}
studied in this paper. Let $h>0$ be a fixed step size, and consider the numerical method defined by
\begin{align}\label{method Cohen}
\bar{X}_{n+1}&=\bar{X}_{n}+hS\bigl(\dfrac{\bar{X}_{n}+\bar{X}_{n+1}}{2}\bigr)
\int_{0}^{1}\nabla I(\bar{X}_{n}+\tau(\bar{X}_{n+1}-\bar{X}_{n}))\,\text d\tau
\nonumber \\[-1.5ex]
\\[-1.5ex]
&\quad+\sum_{r=1}^{m}\Delta \hat{W}_{r}T_{r}\bigl(\dfrac{\bar{X}_{n}+\bar{X}_{n+1}}{2}\bigr)
\int_{0}^{1}\nabla I(\bar{X}_{n}+\tau(\bar{X}_{n+1}-\bar{X}_{n}))\,\text d\tau\nonumber,
\end{align}
where $\Delta \hat{W}_{r}=\sqrt{h}\zeta_{h}^{r}$ with $\zeta_{h}^{r}$ being the
truncation of a $\mathcal{N}(0,1)$-distribution random variable $\xi^r$:
$$
\zeta_h^r=
\begin{cases}
\xi^r, &\text{if}\:\: |\xi^r|\leq A_h,\\
A_h, &\text{if}\:\: \xi^r>A_h, \\
-A_h, &\text{if}\:\: \xi^r<-A_h
\end{cases}
$$
with $A_h:=\sqrt{2k|\ln(h)|}$ for an arbitrary integer $k\geq0$. This choice is motivated by
the fact that standard Gaussian random variables $\Delta {W}_{r}$ are unbounded for arbitrary
small values of $h$, see \cite{Milstein1995} for more details.
Taking $k=2$, we have the following properties \cite{Milstein2002a}
\begin{align}\label{proptrunc}
E(\Delta\hat{W_r})^{2\ell}&\leq Kh^\ell,\quad E(\Delta\hat{W_r})^{2\ell+1}=0,\quad\text{for}\quad\ell\geq0,\nonumber\\
|E((\Delta\hat{W_r})^{2}-(\Delta{W_r})^{2})|&\leq Kh^3,
\quad E(|\Delta\hat{W_r}-\Delta{W_r}|^{2})\leq Kh^3,\\
\quad E|\Delta\hat{W_r}\Delta\hat{W_s}-\Delta{W_r}\Delta{W_s}|^{2}&\leq Kh^3,\nonumber
\end{align}
with a generic constant $K$ that does not depend on $h$. Observe,
that here and in the following the constants $K$ or $C$ may vary from line to line
but are independent on $h$ and $n$.
In fact, it is easy to prove that the integral $\int_{0}^{1}\nabla I(\bar{X}_{n}+
\tau(\bar{X}_{n+1}-\bar{X}_{n}))\,\text d\tau$ in \eqref{method Cohen} is a discrete gradient,
but in general, is not symmetric, see Definition~2.3 in \cite{Zhang2011}.

To conclude this section, we note that the above conservative method
reduces to the numerical scheme proposed in \cite{Cohen2013}
in the case of stochastic Poisson systems, i.\,e., equation \eqref{SG system} with $m=1$
and $T_{1}(x)=cS(x)$ with a real constant $c$.

\section{Properties of the conservative method}\label{sec-prop}
The conservative method \eqref{method Cohen} has been designed
to preserve the invariant $I(x)$ exactly. Indeed, one has the following immediate result.
\begin{proposition}\label{propo1}
The numerical method \eqref{method Cohen} exactly preserves
the invariant, i.\,e., $I(\bar{X}_{n})=I(\bar{X}_{n+1})$ for all $n\geq0$.
\end{proposition}
\begin{proof}
This is similar to the proof of Proposition~3.1 in \cite{Cohen2013}:
the proof follows from the definition of \eqref{method Cohen} and
the skew symmetry of the matrices $S$ and $T_r$.
\end{proof}

If $I(x)$ is of a special form, further interesting properties are enjoyed
by the conservative numerical method \eqref{method Cohen}.
\begin{proposition}
If $I(x)=\dfrac12 x^{T}Cx+d^{T}x$ with $C$ being a symmetric matrix and
$d$ being a constant vector, then method \eqref{method Cohen} reduces
to the stochastic midpoint scheme \cite{Milstein2002a}. Further, it is known that
the stochastic midpoint method preserves all quadratic invariants \cite{Abdulle2012}.
\end{proposition}
\begin{proof}
In the case where $I(x)=\dfrac12 x^{T}Cx+d^{T}x$ we have
\begin{align*}
  \int_{0}^{1}\nabla I(\bar{X}_{n}+\tau(\bar{X}_{n+1}-\bar{X}_{n}))\,\text d\tau&=\int_{0}^{1}
\Big(C(\bar{X}_{n}+\tau(\bar{X}_{n+1}-\bar{X}_{n}))+d\Big)\,\text d\tau\\
&=C\dfrac{\bar{X}_{n}+\bar{X}_{n+1}}{2}+d\\
&=\nabla I(\frac{\bar{X}_{n}+\bar{X}_{n+1}}{2}).
\end{align*}
Substituting this into the method \eqref{method Cohen} and recalling
that $S(x)\nabla I(x)=f(x)$ and $T_{r}(x)\nabla I(x)=g_{r}(x)$, $r=1,\cdots, m$,
we observe that the proposed method \eqref{method Cohen}
reduces to the stochastic midpoint scheme from \cite{Milstein2002a}.
\end{proof}

We next show the following result:
\begin{proposition}
If $I(x)$ is separable, i.\,e. $I(x)=I_{1}(x^{1})+I_{2}(x^{2})+\cdots+I_{d}(x^{d})$ with
$x=(x^1,\ldots,x^d)$, then the conservative method \eqref{method Cohen}
coincides with the symmetric discrete gradient method proposed in \cite{Zhang2011}
(for a one-dimensional Brownian motion).
\end{proposition}
\begin{proof}
Since $I(x)$ is separable, we have $I_{1},\cdots,I_{d}$ such that
\[I(x)=I_{1}(x^{1})+I_{2}(x^{2})+\cdots+I_{d}(x^{d}).\]
It then follows that the $k$th component of $\int_{0}^{1}\nabla I(\bar{X}_{n}+\tau(\bar{X}_{n+1}-\bar{X}_{n}))\,\text d\tau$ reads
\begin{align*}
&\Bigl(\int_{0}^{1}\nabla I(\bar{X}_{n}+\tau(\bar{X}_{n+1}-\bar{X}_{n}))\,\text d\tau\Bigr)^k\\
&=\int_{0}^{1}\nabla I_{k}(\bar{X}_{n}^{k}+\tau(\bar{X}_{n+1}^{k}-\bar{X}_{n}^{k}))\,\text d\tau\\
&=\int_{0}^{1}\dfrac{1}{\bar{X}_{n+1}^{k}-\bar{X}_{n}^{k}}
\dfrac{\text d}{\text d\tau}I_{k}(\bar{X}_{n}^{k}+\tau(\bar{X}_{n+1}^{k}-\bar{X}_{n}^{k}))\,\text d\tau\\
&=\dfrac{I_{k}(\bar{X}_{n+1}^{k})-I_{k}(\bar{X}_{n}^{k})}{\bar{X}_{n+1}^{k}-\bar{X}_{n}^{k}}\\
&=\Big(\bar{\nabla}I(\bar{X}_{n},\bar{X}_{n+1})\Big)^{k},
\end{align*}
where $\bar{\nabla}I(\bar{X}_{n},\bar{X}_{n+1})$
is the symmetric discrete gradient defined in \cite{Zhang2011}.
Inserting this expression in the definition of
the conservative method \eqref{method Cohen}, one notice that the proposed method
reduces to the discrete gradient method from \cite{Zhang2011} in case of a separable
conserved quantity $I$.
\end{proof}

\subsection{Mean-square order}
For stochastic Poisson systems, i.\,e., equation \eqref{SG system} with $m=1$,
and $T_{1}(x)=cS(x)$ with a real constant $c$, the authors of \cite{Cohen2013} show
that the mean-square convergence order of the numerical scheme \eqref{method Cohen} is $1$.
For general stochastic differential equations with a conserved quantity as studied in the present work,
we now show that the mean-square convergence order of the conservative method remains $1$ under
the condition of commutative noise. We recall this condition for equation \eqref{eq1}
\[\Lambda_{i}g_{r}(x)=\Lambda_{r}g_{i}(x),\quad \text{for }\quad i,r=1,\cdots, m,\]
with the operator $\Lambda_{i}:=(g_{i},\;\frac{\partial}{\partial x})=\sum_{j=1}^{d}g_{i}^{j}\frac{\partial}{\partial x^{j}}$.


\begin{theorem}\label{theo cohen}
Consider problem \eqref{eq1} with a scalar invariant $I$ discretised
by the conservative numerical method \eqref{method Cohen} with step size $h$.
Assume that the matrix-functions $S,T_{r}\in C_{b}^{2}(\R^{d},\R^{d\times d})$,
that $\nabla I$ satisfies a global Lipschitz condition and has
uniformly bounded first and second derivatives.
Assume further that the noises satisfy the commutative conditions.
Then there exist a constant $K>0$ (independent of $n$ and $h$)
such that the following error estimate holds , for $n=0,1,\cdots, N$ with $N=[T/h]$,
\[
(E|X(t_{n})-\bar{X}_{n}|^{2})^{\frac12}\leq Kh,\quad
\text{for all $h$ sufficiently small.}
\]
Here, we recall that $X(t)$ denotes the exact solution of \eqref{eq1}
and $\bar{X}_{n}$ the numerical one on the time interval $[0,T]$.
I.\,e., the numerical method \eqref{method Cohen} is of first order in the mean-square
convergence sense.
\end{theorem}
\begin{proof}
The main idea of the proof is to compare our conservative scheme
to Milstein's scheme applied to the converted It\^o SDE and use Lemma~2.1 in \cite{mrt01}
to ensure that the conservative
scheme has mean-square order of convergence one.
In order to do this, we first rewrite the one-step approximation
scheme \eqref{method Cohen} (starting at $x$) by
\begin{align}\label{one step Cohen}
\bar{X}&=x+hS(\dfrac{x+\bar{X}}{2})\int_{0}^{1}\nabla I(x+\tau(\bar{X}-x))\,\text d\tau\nonumber\\[-1.5ex]
\\[-1.5ex]
&\quad+\sum_{r=1}^{m}\Delta \hat{W}_{r}T_{r}(\dfrac{x+\bar{X}}{2})\int_{0}^{1}\nabla I(x+\tau(\bar{X}-x))\,\text d\tau.\nonumber
\end{align}
Let $\tilde{X}$ be the corresponding one-step approximation of Milstein's method (starting at $x$)
applied to \eqref{SG system} (converted to an It\^o SDE),
\begin{align}\label{one step Milstein}
\tilde{X}&=x+hS(x)\nabla I(x)+\sum_{r=1}^{m}\Delta W_{r} T_{r}(x)\nabla I(x)\nonumber \\[-1.5ex]
\\[-1.5ex]
&\quad+\sum_{i=1}^{m-1}\sum_{r=i+1}^{m}\Lambda_{i}(T_{r}(x)\nabla I(x))\Delta W_{i}\Delta W_{r}
+\dfrac12\sum_{r=1}^{m}\Lambda_{r}(T_{r}(x)\nabla I(x))(\Delta W_{r})^{2}\nonumber.
\end{align}
From \cite{Milstein1995}, we know that Milstein's method
is of mean-square order $1$ under the condition of our theorem,
in particular, if $X_{t,x}(t+h)$ denotes the exact solution of \eqref{SG system}
on $[t,t+h]$ starting at $x$, then
\begin{equation*}
|E(\tilde{X}-X_{t,x}(t+h))|\leq K(1+|x|^2)^{1/2}h^2,\quad
(E|\tilde{X}-X_{t,x}(t+h)|^{2})^{1/2}\leq K(1+|x|^2)^{1/2}h^{\frac32}.
\end{equation*}
Thus, in order to show that the numerical scheme \eqref{method Cohen} is
of mean-square order $1$ as well, using Lemma~2.1 in \cite{mrt01}, we will prove that
\begin{equation*}
|E(\bar{X}-\tilde{X})|=\mathcal{O}(h^{2}),\quad (E|\bar{X}-\tilde{X}|^{2})^{1/2}=\mathcal{O}(h^{\frac32}),
\end{equation*}
where, here and in the following, the constants in the $\mathcal{O}(\cdot)$ notations
may depend on the starting point $x$ for the schemes but are independent
of $h$ and $n$. For any $k=1,2,\cdots, d$,
the corresponding component equation of \eqref{one step Milstein} is
\begin{align*}
\tilde{X}^{k}&=x^{k}+\sum_{i=1}^{d}(S^{ki}\partial_{i}I)h+
\sum_{r=1}^{m}\sum_{i=1}^{d}(T_{r}^{ki}\partial_{i}I)\Delta W_{r}\\
&\quad+\dfrac12 \sum_{r=1}^{m} \sum_{i,j=1}^{d}(\partial_{j}T_{r}^{ki}
\partial_{i}I+T_{r}^{ki}\partial_{ij}I)
(\sum_{l=1}^{d}T_{r}^{jl}\partial_{l}I)(\Delta W_{r})^{2}\\
&\quad+\sum_{i=1}^{m-1}\sum_{r=i+1}^{m}
\sum_{l,j=1}^{d}(\partial_{j}T_{r}^{kl}
\partial_{l}I+T_{r}^{kl}\partial_{lj}I)
(\sum_{l=1}^{d}T_{i}^{jl}\partial_{l}I)(\Delta W_{i})(\Delta W_{r}).
\end{align*}
We next develop an expansion for the $k$th component equation
of \eqref{one step Cohen}. By assumptions, using deterministic Taylor expansions,
there exists $0<\theta<1$ (below $\theta$
maybe differ from line to line) such that
\begin{equation*}
S^{ki}(\dfrac{x+\bar{X}}{2})=S^{ki}(x)+\dfrac12 \sum_{j=1}^{d}\partial_{j}S^{ki}(x)\Delta^{j}+R_{S},
\end{equation*}
where $\Delta^{j}:=\bar{X}^{j}-x^{j}$ and the remainder term is given by
\begin{equation*}
R_{S}=\dfrac{1}{8}\sum_{m,n=1}^{d}
\partial_{mn}S^{ki}(x+\theta\dfrac{\bar{X}-x}{2})\Delta^{m}\Delta^{n}.
\end{equation*}
For the matrix-functions $T_{r}$, we have a similar expansion
\begin{equation*}
T_{r}^{ki}(\dfrac{x+\bar{X}}{2})=T_{r}^{ki}(x)+
\dfrac12 \sum_{j=1}^{d}\partial_{j}T_{r}^{ki}(x)\Delta^{j}+R_{T_{r}},
\end{equation*}
where the remainder term reads
\begin{equation*}
R_{T_{r}}=\dfrac{1}{8}
\sum_{m,n=1}^{d}\partial_{mn}T_{r}^{ki}(x+\theta\dfrac{\bar{X}-x}{2})\Delta^{m}\Delta^{n}.
\end{equation*}
Similarly, the component expansion of $\nabla I(x+\tau(\bar{X}-x))$ reads
\begin{equation*}
\partial_{i}I(x+\tau(\bar{X}-x))=\partial_{i}I(x)+
\tau\sum_{j=1}^{d}\partial_{ij}I(x)\Delta^{j}+R_{I},
\end{equation*}
with $R_{I}=\dfrac{\tau^{2}}{2}\sum_{j,k=1}^{d}
\partial_{ijk}I(x+\theta \tau(\bar{X}-x))\Delta^{j}\Delta^{k}.$

Substituting these expansions into the $k$th
component equation of \eqref{one step Cohen}, we obtain
\begin{align}\label{k_com_Num}
\bar{X}^{k}&=x^{k}+\sum_{i=1}^{d}S^{ki}\partial_{i}Ih+
\sum_{r=1}^{m}\sum_{i=1}^{d}T_r^{ki}\partial_{i}I\Delta \hat{W}_{r}\nonumber\\[-1.5ex]
\\[-1.5ex]
&\quad+\dfrac12 \sum_{r=1}^{m}\sum_{l,j=1}^{d}\Big(\partial_{j}T_{r}^{kl}\partial_{l}I+T_{r}^{kl}
\partial_{lj}I\Big)\Delta^{j}\Delta \hat{W}_{r}+R_{1},\nonumber
\end{align}
where
\begin{align*}
R_{1}&=\sum_{i=1}^{d}S^{ki}\Big(\int_{0}^{1}\partial_{i}I(x+\tau(\bar{X}-x))\,\text d\tau-\partial_{i}I(x)\Big)h\\
&\quad+\sum_{i=1}^{d}\Big(\dfrac12\sum_{j=1}^{d}\partial_{j}S^{ki}\Delta^{j}+R_{S}\Big)
\int_{0}^{1}\partial_{i}I(x+\tau(\bar{X}-x))\,\text d\tau h\\
&\quad+\dfrac12 \sum_{r=1}^{m}\sum_{i,j=1}^{d}\partial_{j}T_{r}^{ki}\Delta^{j}
\Big(\int_{0}^{1}\partial_{i}I(x+\tau(\bar{X}-x))\,\text d\tau-\partial_{i}I(x)\Big)\Delta \hat{W}_{r}\\
&\quad+\sum_{r=1}^{m}\sum_{i=1}^{d}R_{T_{r}}\int_{0}^{1}\partial_{i}I(x+\tau(\bar{X}-x))\,\text d\tau\Delta \hat{W}_{r}
+\sum_{r=1}^{m}\sum_{i=1}^{d}T_{r}^{ki}\int_{0}^{1}R_{I}\,\text d\tau\Delta \hat{W}_{r}.
\end{align*}
Since the noises are commutative, i.\,e., for $k=1,\cdots,d$ and $i,r=1,\cdots,m$,
\[\sum_{l,j=1}^{d}(\partial_{j}T_{r}^{kl}\partial_{l}I+T_{r}^{kl}\partial_{lj}I)(\sum_{l=1}^{d}T_{i}^{jl}\partial_{l}I)
=\sum_{l,j=1}^{d}(\partial_{j}T_{i}^{kl}\partial_{l}I+T_{i}^{kl}\partial_{lj}I)(\sum_{l=1}^{d}T_{r}^{jl}\partial_{l}I), \]
we have, after rearranging terms in the sums,
\begin{align*}
&\dfrac12\sum_{r=1}^{m}\sum_{l,j=1}^{d}(\partial_{j}T_{r}^{kl}\partial_{l}I+T_{r}^{kl}\partial_{lj}I)
   (\sum_{i=1}^{m}\sum_{l=1}^{d}T_{i}^{jl}\partial_{l}I{\Delta}\hat{W}_{i}){\cdot \Delta\hat{W}_r}\\
&=\dfrac12 \sum_{r=1}^{m} \sum_{i,j=1}^{d}(\partial_{j}T_{r}^{ki}\partial_{i}I+T_{r}^{ki}\partial_{ij}I)(\sum_{l=1}^{d}T_{r}^{jl}\partial_{l}I)(\Delta \hat{W}_{r})^{2}\\
&\quad+\sum_{i=1}^{m-1}\sum_{r=i+1}^{m}\sum_{l,j=1}^{d}(\partial_{j}T_{r}^{kl}\partial_{l}I+T_{r}^{kl}\partial_{lj}I)(\sum_{l=1}^{d}T_{i}^{jl}\partial_{l}I)
(\Delta \hat{W}_{i})(\Delta \hat{W}_{r}).
\end{align*}
Substituting it into \eqref{k_com_Num}, we obtain
\begin{align}\label{k_cons_rest}
\bar{X}^{k}&=x^{k}+\sum_{i=1}^{d}S^{ki}\partial_{i}Ih+
\sum_{r=1}^{m}\sum_{i=1}^{d}T_r^{ki}\partial_{i}I\Delta \hat{W}_{r}\nonumber\\
&\quad+\dfrac12 \sum_{r=1}^{m} \sum_{i,j=1}^{d}(\partial_{j}T_{r}^{ki}\partial_{i}I+T_{r}^{ki}\partial_{ij}I)
(\sum_{l=1}^{d}T_{r}^{jl}\partial_{l}I)(\Delta \hat{W}_{r})^{2}\nonumber\\[-1.5ex]
\\[-1.5ex]
&\quad+\sum_{i=1}^{m-1}\sum_{r=i+1}^{m}\sum_{l,j=1}^{d}(\partial_{j}T_{r}^{kl}
\partial_{l}I+T_{r}^{kl}\partial_{lj}I)(\sum_{l=1}^{d}T_{i}^{jl}\partial_{l}I)
(\Delta \hat{W}_{i})(\Delta \hat{W}_{r})\nonumber\\
&\quad+R_{1}+R_{2},\nonumber
\end{align}
where
\begin{align*}
R_{2}=\dfrac12\sum_{r=1}^{m}\sum_{l,j=1}^{d}(\partial_{j}T_{r}^{kl}\partial_{l}I+T_{r}^{kl}\partial_{lj}I)
\Big(\Delta^{j}-\sum_{i=1}^{m}\sum_{l=1}^{d}T_{i}^{jl}\partial_{l}I\Delta \hat{W}_{i}\Big)\Delta \hat{W}_{r}.
\end{align*}
Under the assumptions that $S,T_{r}\in C_{b}^{2}(\R^{d},\R^{d\times d})$,
the ones on the invariant $I$,
and due to the properties of $\Delta \hat{W}_{r}$, see \eqref{proptrunc},
we derive the following estimation from equation \eqref{one step Cohen}
\begin{align}\label{est-delta}
(E(\Delta^{i})^{2\ell})^{\frac{1}{2\ell}}\leq (E|\Delta|^{2\ell})^{\frac{1}{2\ell}}\leq Kh^{\frac12},\quad \ell\geq 1,
\end{align}
where $\Delta=(\Delta^i)_{i=1}^d$. Further, we know that $(E|R_{1}|^{2})^{\frac12}=\mathcal{O}(h^{\frac32})$.
These estimations and equation \eqref{k_com_Num} give us $|E(\Delta^{i})|=\mathcal{O}(h)$.
The estimation $|E(R_{1})|=\mathcal{O}(h^{2})$ follows from substituting $\Delta^{j}$
into the last three terms of $R_{1}$ and from the properties of $\Delta \hat{W}_{r}$ in \eqref{proptrunc}.
Similarly we get $(E|R_{2}|^{2})^{\frac12}=\mathcal{O}(h^{\frac32})$ and $|E(R_{2})|=\mathcal{O}(h^{2})$.
We now compare our conservative scheme, see also \eqref{k_cons_rest}, and Milstein's method
\begin{align*}
\rho^{k}&:=\bar{X}^{k}-\tilde{X}^{k}\\
&=\dfrac12 \sum_{r=1}^{m} \sum_{i,j=1}^{d}(\partial_{j}T_{r}^{ki}\partial_{i}I+T_{r}^{ki}\partial_{ij}I)(\sum_{l=1}^{d}T_{r}^{jl}\partial_{l}I)
((\Delta \hat{W}_{r})^{2}-(\Delta W_{r})^{2})\\
&+\sum_{i<r}\sum_{l,j=1}^{d}(\partial_{j}T_{r}^{kl}\partial_{l}I+T_{r}^{kl}\partial_{lj}I)
(\sum_{l=1}^{d}T_{i}^{jl}\partial_{l}I)((\Delta \hat{W}_{i})(\Delta \hat{W}_{r})-(\Delta W_{i})(\Delta W_{r}))\\
&+\sum_{r=1}^{m}\sum_{i=1}^{d}T_{r}^{ki}\partial_{i}I(\Delta \hat{W}_{r}-\Delta W_{r})+R_{1}+R_{2}.
\end{align*}
And obtain the estimations
\begin{equation*}
|E(\rho)|=\mathcal{O}(h^{2}),\quad E|\rho|^{2}=\mathcal{O}(h^{3}),
\end{equation*}
with the vector $\rho=(\rho^k)_{k=1}^d$. Lemma~2.1 in \cite{mrt01} thus implies that the conservative scheme
\eqref{method Cohen} is of mean-square order $1$ and thus completes the proof.
\end{proof}

\begin{remark}
In the above proof, we need commutative noise. Without this condition,
the mean-square convergence order of the conservative method \eqref{method Cohen}
is only $\frac12$. However, as we will see next, the commutativity condition is no more needed
to get weak order of convergence $1$. It is meaningful to construct
high weak order method, see \cite{Abdulle2012,Milstein1995,Kloeden1992} for instance.
\end{remark}

%
%
%
%
\subsection{Weak order}
We will now show that the conservative numerical method \eqref{method Cohen} has weak convergence order $1$.
Before that, we point out that, for sufficiently large $\ell$, $E|\bar{X}_{n}|^{2\ell}$ exist
and are uniformly bounded for all $n=0,1,\cdots,N$ according to the proof of Theorem~\ref{theo cohen}
and Lemma~2.2 in \cite[Sect. 2.2.1]{Milstein1995}.
\begin{theorem}\label{theo_weak_1}
Assume that the functions $S, T_{r}\in C_{b}^{4}(\R^{d},\R^{d\times d})$ and
$\nabla I$ satisfies a global Lipschitz condition and has uniformly bounded derivatives
from first to forth order.
 Let further $\psi\in C_{P}^{4}(\R^{d},\R)$.
Then the following inequality holds
\begin{equation*}
|E\psi(X(t_{n}))-E\psi(\bar{X}_{n})|\leq Kh,
\end{equation*}
for all $n=0,1,\cdots, N$  with a positive constant $K$ independent of $n$ and $h$ (small enough).
I.\,e., the conservative method \eqref{method Cohen} has order
of accuracy $1$ in the sense of weak approximations.
\end{theorem}
\begin{proof}
To show that the weak order of accuracy of our numerical method is $p=1$,
we will use the main theorem on convergence of weak approximations
\cite[Theorem~2.1, Sect. 2.2.1]{Milstein1995}, see also \cite[Theorem~14.5.2]{Kloeden1992},
and prove the following estimates
\begin{equation}\label{est1}
\Big|E\Big(\prod_{j=1}^{s}\Delta^{i_{j}}-\prod_{j=1}^{s}\bar{\Delta}^{i_{j}}\Big)\Big|
\leq K(x)h^{p+1},\quad s=1,\cdots,2p+1,
\end{equation}
and
\begin{equation}\label{est2}
E\prod_{j=1}^{2(2p+2)}|\bar{\Delta}^{i_{j}}|\leq K(x)h^{2p+2},
\end{equation}
where $K(x)$ is some function with polynomial growth
and we use the notations $\Delta^i:=X^i-x^i$ and $\bar{\Delta}^{i}:=\bar{X}^i-x^i$
with $X^i$ being the $i$th component of the exact solution of equation \eqref{eq1}
starting from $x$, and $\bar{X}$ being its numerical approximation
(given by \eqref{method Cohen} in our case). From the proof of Theorem~\ref{theo cohen}
and the use of Cauchy-Schwarz inequality,
one easily obtains estimation \eqref{est2}.
Below we will show that \eqref{est1} holds for $p=1$.

The $k$th component of $X(t)$ satisfies the It\^o SDE
\begin{align*}
\text dX^{k}&=\sum_{i=1}^{d}S^{ki}\partial_{i}I\,\text dt+\dfrac{1}{2}
\sum_{r=1}^{m}\sum_{i,j=1}^{d}(\partial_{j}T_{r}^{ki}\partial_{i}I+T_{r}^{ki}\partial_{ij}I)
(\sum_{l=1}^{d}T_{r}^{jl}\partial_{l}I)\,\text dt\nonumber\\
&\quad+\sum_{r=1}^{m}\sum_{i=1}^{d}T_{r}^{ki}\partial_{i}I\,\text dW_{r}(t).
\end{align*}
To simplify the notations, we let
\[a^{k}=\sum_{i=1}^{d}S^{ki}\partial_{i}I+\dfrac{1}{2}\sum_{r=1}^{m}\sum_{i,j=1}^{d}(\partial_{j}T_{r}^{ki}\partial_{i}I+T_{r}^{ki}\partial_{ij}I)(\sum_{l=1}^{d}T_{r}^{jl}\partial_{l}I)\]
and $g_{r}^{k}=\sum_{i=1}^{d}T_{r}^{ki}\partial_{i}I$. Then
\begin{equation}\label{k_com_X}
X_{t,x}^{k}(t+h)=x^{k}+\int_{t}^{t+h}a^{k}(X(s))\,\text ds+
\sum_{r=1}^{m}\int_{t}^{t+h}g_{r}^{k}(X(s))\,\text dW_{r}(s).
\end{equation}
We now prove \eqref{est1} for $s=1$. From the proof of Theorem~\ref{theo cohen}, we have the expansion \eqref{k_cons_rest} of the conservative method $\bar{X}^{k}$. Compare it with equation \eqref{k_com_X}, we have

\begin{align*}
|E(\Delta^{k}-\bar{\Delta}^{k})|=\Big|E\int_{t}^{t+h}a^{k}(X(s))\,\text ds-a^k(x)h-E(R_{1}+R_{2})-E(R_{3})\Big|,
\end{align*}
where \[E(R_{3})=\dfrac{1}{2}\sum_{r=1}^{m}\sum_{i,j=1}^{d}
(\partial_{j}T_{r}^{ki}\partial_{i}I+T_{r}^{ki}\partial_{ij}I)
(\sum_{l=1}^{d}T_{r}^{jl}\partial_{l}I)E[(\Delta \hat{W}_{r})^{2}-(\Delta W_{r})^{2}].\]
We know that (recalling that we use truncated random variables, see Section~\ref{sec-meth})
\[|E(R_{1}+R_{2})|\leq Kh^{2},\quad |E(R_{3})|\leq Kh^{3}.\]
Hence
\begin{align*}
|E(\Delta^{k}-\bar{\Delta}^{k})|&\leq \Big|E\int_{t}^{t+h}a^{k}(X(s))\,\text ds-a^k(x)h\Big|+Kh^{2}\\
&\leq Kh^{2}+\int_{t}^{t+h}\sum_{n_{1}=1}^{d}\dfrac{\partial a^{k}(x)}{\partial x_{n_{1}}}|E\Delta^{n_{1}}(s)|\,\text ds\\
&\quad+\int_{t}^{t+h}\sum_{n_{1},n_{2}=1}^{d}\Big|E\dfrac{\partial^{2}a^{k}(x^{\theta})}{\partial x_{n_{1}}\partial x_{n_{2}}}\Delta^{n_{1}}(s)\Delta^{n_{2}}(s)\Big|\,\text ds\\
&\leq Kh^{2}.
\end{align*}
This proves equality \eqref{est1} for $s=1$.
We next show that \eqref{est1} holds for $s=2$.
By definition of $\Delta^{j}$ and a use of It\^o's isometry, we have
\begin{align*}
E(\Delta^{i_{1}}\Delta^{i_{2}})&=E\Big\{\Big(\int_{t}^{t+h}a^{i_{1}}(X(s))\,\text ds+
\sum_{r=1}^{m}\int_{t}^{t+h}g_{r}^{i_{1}}(X(s))\,\text dW_{r}(s)\Big)\nonumber\\
&\Big(\int_{t}^{t+h}a^{i_{2}}(X(s))\,\text ds+\sum_{r=1}^{m}\int_{t}^{t+h}g_{r}^{i_{2}}(X(s))\,\text dW_{r}(s)\Big)\Big\}\nonumber\\
&=E\int_{t}^{t+h}a^{i_{1}}(X(s))\,\text ds\int_{t}^{t+h}a^{i_{2}}(X(s))\,\text ds\nonumber\\
&\quad+\sum_{r=1}^{m}E\int_{t}^{t+h}a^{i_{1}}(X(s))\,\text ds\int_{t}^{t+h}g_{r}^{i_{2}}(X(s))\,\text dW_{r}(s)\nonumber\\
&\quad+\sum_{r=1}^{m}E\int_{t}^{t+h}g_{r}^{i_{1}}(X(s))\,\text dW_{r}(s)\int_{t}^{t+h}a^{i_{2}}(X(s))\,\text ds\nonumber\\
&\quad+\sum_{r=1}^{m}E\int_{t}^{t+h}g_{r}^{i_{1}}(X(s))g_{r}^{i_{2}}(X(s))\,\text ds.
\end{align*}
By definition of $\bar{\Delta}^{j}$, we get
\begin{align*}
E(\bar{\Delta}^{i_{1}}\bar{\Delta}^{i_{2}})=a^{i_{1}}(x)a^{i_{2}}(x)h^{2}+\sum_{r=1}^{m}g_{r}^{i_{1}}(x)g_{r}^{i_{2}}(x)h+\hat{R}
\end{align*}
with $|E(\hat{R})|\leq Kh^{2}$.

Since
\begin{align*}
&\Big|E\int_{t}^{t+h}a^{i_{1}}(X(s))\,\text ds\int_{t}^{t+h}g_{r}^{i_{2}}(X(s))\,\text dW_{r}(s)\Big|\\
&=\Big|E\Big(\int_{t}^{t+h}(a^{i_{1}}(X(s))-a^{i_{1}}(x))\,\text ds+a^{i_{1}}(x)h\Big)\\
&\quad\Big(\int_{t}^{t+h}(g_{r}^{i_{2}}(X(s))-g_{r}^{i_{2}}(x))\,\text dW_{r}(s)+g_{r}^{i_{2}}(x)\Delta W_{r}\Big)\Big|\\
&=\Big|E\int_{t}^{t+h}(a^{i_{1}}(X(s))-a^{i_{1}}(x))\,\text ds
\int_{t}^{t+h}(g_{r}^{i_{2}}(X(s))-g_{r}^{i_{2}}(x))\,\text dW_{r}(s)\\
&\quad+g_{r}^{i_{2}}(x)E\Delta W_{r}\int_{t}^{t+h}(a^{i_{1}}(X(s))-a^{i_{1}}(x))\,\text ds\Big|\\
&\leq Kh^{2}
\end{align*}
and, by a Taylor expansions,
\[\Big|E\int_{t}^{t+h}g_{r}^{i_{1}}(X(s))g_{r}^{i_{2}}(X(s))\,\text ds-g_{r}^{i_{1}}(x)g_{r}^{i_{2}}(x)h\Big|\leq Kh^{2},\]
we obtain that
\[|E(\Delta^{i_{1}}\Delta^{i_{2}}-\bar{\Delta}^{i_{1}}\bar{\Delta}^{i_{2}})|=\mathcal{O}(h^{2}).\]
We finally prove that equality \eqref{est1} holds for $s=3$. As above, if we write down the expressions for  $E(\Delta^{i_{1}}\Delta^{i_{2}}\Delta^{i_{3}})$ and
$E(\bar\Delta^{i_{1}}\bar\Delta^{i_{2}}\bar\Delta^{i_{3}})$, we will observe
that we only have to estimate the following term:
\begin{align*}
&\Big|E\int_{t}^{t+h}g_{r_{1}}^{i_{1}}(X(s))\,\text dW_{r_{1}}
\int_{t}^{t+h}g_{r_{2}}^{i_{2}}(X(s))\,\text dW_{r_{2}}\int_{t}^{t+h}g_{r_{3}}^{i_{3}}(X(s))\,\text dW_{r_{3}}\Big|\\
&=\Big|E\Big(\int_{t}^{t+h}(g_{r_{1}}^{i_{1}}(X(s))-g_{r_{1}}^{i_{1}}(x))\,\text dW_{r_{1}}(s)+g_{r_{1}}^{i_{1}}(x)\Delta W_{r_{1}}\Big)\\
&\quad\Big(\int_{t}^{t+h}(g_{r_{2}}^{i_{2}}(X(s))-g_{r_{2}}^{i_{2}}(x))\,\text dW_{r_{2}}(s)+g_{r_{2}}^{i_{2}}(x)\Delta W_{r_{2}}\Big)\\
&\quad\Big(\int_{t}^{t+h}(g_{r_{3}}^{i_{3}}(X(s))-g_{r_{3}}^{i_{3}}(x))\,\text dW_{r_{3}}(s)+g_{r_{3}}^{i_{3}}(x)\Delta W_{r_{3}}\Big)\Big|\\
&\leq Kh^{2}.
\end{align*}
Therefore,
\[|E(\Delta^{i_{1}}\Delta^{i_{2}}\Delta^{i_{3}}-\bar{\Delta}^{i_{1}}\bar{\Delta}^{i_{2}}\bar{\Delta}^{i_{3}})|=\mathcal{O}(h^{2}).\]
Thus we complete the proof of this theorem.
\end{proof}

\section{Quadrature rule}\label{sec-qf}
In this section, we will investigate the use of a quadrature formula $(c_{i},b_{i})_{i=1}^{D}$
$$
\int_{0}^{1}f(\tau)\,\text d\tau\approx\sum_{i=1}^{D}b_{i}f(c_{i})
$$
to approximate the integral present in the conservative numerical method \eqref{method Cohen}.
In this case, we obtain the following numerical approximation
\begin{align}\label{method Cohen app}
\widehat{X}_{n+1}&=\widehat{X}_{n}+hS(\dfrac{\widehat{X}_{n}+\widehat{X}_{n+1}}{2})\sum_{i=1}^{D}b_{i}\nabla I(\widehat{X}_{n}+c_{i}(\widehat{X}_{n+1}-\widehat{X}_{n}))\nonumber\\[-1.5ex]
\\[-1.5ex]
&\quad+\sum_{r=1}^{m}\Delta\hat{W_r}T_{r}(\dfrac{\widehat{X}_{n}+\widehat{X}_{n+1}}{2})\sum_{i=1}^{D}b_{i}\nabla I(\widehat{X}_{n}+c_{i}(\widehat{X}_{n+1}-\widehat{X}_{n})).\nonumber
\end{align}
Second moments of such numerical approximations are seen to be bounded as this was done in the previous section.

We first investigate the effect of the use of a quadrature formula on the conservation of $I$.
\begin{proposition}\label{propo2}
The numerical scheme \eqref{method Cohen app} exactly preserves
polynomial conserved quantity $I(x)$ of degree $\nu \leq q$, where $q$
is the order of the quadrature formula. On the other hand,
in the case where $S,T_r\in C_{b}(\R^{d},\R^{d\times d})$ and $\nabla I\in C^{q}_{b}(\R^{d},\R^{d})$,
then one has $E(I(\widehat{X}_{n+1})-I(\widehat{X}_{n}))^{2}=\mathcal{O}(h^{q+1})$.
\end{proposition}
\begin{proof}
The proof of the first statement results from the definition of the order of
a quadrature formula.

On the other hand, from equation \eqref{method Cohen app}, we know that
\begin{equation}\label{4}
  E|\widehat{X}_{n+1}-\widehat{X}_{n}|^{2\ell}=\mathcal{O}(h^{\ell}).
\end{equation}
The expression for the error in the conserved quantity reads
\begin{align*}
  I(\widehat{X}_{n+1})-I(\widehat{X}_{n})
  &=\Big(\delta I\Big)^{T}
     S(\dfrac{\widehat{X}_{n}+\widehat{X}_{n+1}}{2})\Big(\sum_{i=1}^{D}b_{i}\nabla I(\sigma(c_{i} h))\Big)h\\
     &\quad+\sum_{r=1}^{m}\Big(\delta I\Big)^{T}
     T_{r}(\dfrac{\widehat{X}_{n}+\widehat{X}_{n+1}}{2})\Big(\sum_{i=1}^{D}b_{i}\nabla I(\sigma(c_{i} h))\Big)\Delta\hat{W}_{r},
\end{align*}
where we use the notations $\displaystyle\delta I=\int_{0}^{1}\nabla I(\sigma(\tau h))\,\text d\tau
-\sum_{i=1}^{D}b_{i}\nabla I(\sigma(c_{i} h))$
and $\sigma(\tau h)=\widehat{X}_{n}+\tau(\widehat{X}_{n+1}-\widehat{X}_{n})$.

Since the order of the first term is higher than the second one,
we only need to estimate the second term. Using $S,T_r\in C_{b}(\R^{d},\R^{d\times d})$ and $I\in C^{q+1}_{b}(\R^{d},\R)$,
the second statement follows from the following estimates
\begin{align*}
  &E\Big[\Big(\delta I\Big)^{T}
   T_{r}(\frac{\widehat{X}_{n}+\widehat{X}_{n+1}}{2})\Big(\sum_{i=1}^{D}b_{i}\nabla I(\sigma(c_{i} h))\Big)\Delta\hat{W}_{r}\Big]^{2}\\
   &\leq KhE|\delta I|^{2}
   \leq Kh\Big(E\Big(|\dfrac{\partial^{q+1}I(\theta)}{\partial x^{q+1}}|^{2}|\widehat{X}_{n+1}-\widehat{X}_{n}|^{2q}\Big)\Big)\\
   &\leq Kh\Big(E|\widehat{X}_{n+1}-\widehat{X}_{n}|^{2q}\Big)
   \leq Kh^{q+1}.
\end{align*}
Here, $\theta$ denotes a real number appearing in the expression of the remainder
of the Taylor expansions up to order $q$ of $\nabla I$ and the last inequality
follows from the estimations \eqref{4}.
\end{proof}

To investigate the effect of the use of a quadrature formula on the convergence orders of the scheme,
we start with the case where $S$ and $T_r$ are constant skew symmetric matrices.
Then the numerical approximation \eqref{method Cohen app} reads
\begin{align}\label{eq2}
\widehat{X}_{n+1}&=\widehat{X}_{n}+h\sum_{i=1}^{D}S\nabla I(\widehat{X}_{n}+c_{i}(\widehat{X}_{n+1}-\widehat{X}_{n}))b_{i}\nonumber\\[-1.5ex]
\\[-1.5ex]
&\quad +\sum_{r=1}^{m}\Delta\hat{W}_{r}\sum_{i=1}^{D}T_{r}\nabla I(\widehat{X}_{n}+c_{i}(\widehat{X}_{n+1}-\widehat{X}_{n}))b_{i}.\nonumber
\end{align}
Denote $Y_{i}=\widehat{X}_{n}+c_{i}(\widehat{X}_{n+1}-\widehat{X}_{n})$, then we have
\begin{align*}
\widehat{X}_{n+1}&=\widehat{X}_{n}+h\sum_{i=1}^{D}S\nabla I(Y_{i})b_{i}+\sum_{r=1}^{m}\Delta\hat{W}_{r}\sum_{i=1}^{D}T_{r}\nabla I(Y_{i})b_{i}\\
&=\widehat{X}_{n}+h\sum_{i=1}^{D}f(Y_{i})b_{i}+\sum_{r=1}^{m}\Delta\hat{W}_{r}\sum_{i=1}^{D}g_{r}(Y_{i})b_{i}
\end{align*}
and
\begin{align*}
Y_{i}&=\widehat{X}_{n}+c_{i}\Big[h\sum_{j=1}^{D}f(Y_{j})b_{j}+\sum_{r=1}^{m}\Delta\hat{W}_{r}\sum_{j=1}^{D}g_{r}(Y_{j})b_{j}\Big]\\
&=\widehat{X}_{n}+h\sum_{j=1}^{D}c_{i}b_{j}f(Y_{j})+\sum_{r=1}^{m}\Delta\hat{W}_{r}\sum_{j=1}^{D}
c_{i}b_{j}g_{r}(Y_{j}).
\end{align*}
This is nothing but an implicit $D$-stage stochastic Runge-Kutta
method with Butcher tableau 

\begin{center}
\begin{tabular}{c||c|c|c}
$c$ & $cb^T$ & $\cdots$ & $cb^T$\\\hline
  & $b^T$ & $\cdots$ & $b^T$\\
\end{tabular}\\
\quad\quad{$\underbrace{\rule{0.2\columnwidth}{0pt}}_{m\:{times}}$}
\end{center}

Using now a quadrature formula $(c_{i},b_{i})_{i=1}^{D}$ of order bigger than $1$, we have
\[1=\int_{0}^{1}1\,\text d\tau=\sum_{i=1}^{D}b_{i} \quad\text{ and }\quad \dfrac12=\int_{0}^{1}\tau\,\text d\tau =\sum_{i=1}^{D}c_{i}b_{i}.\]
This implies that the mean-square order of the method \eqref{eq2} is $1$ (in the commutative case) using results from \cite{bb00}
and the weak order is also $1$ using results from \cite{Rossler2004}.

We next present the result for non-constant matrices $S(x)$ and $T_r(x)$.
\begin{theorem}\label{theo cohen2}
Let $q$ be the order of the quadrature formula $(c_{i},b_{i})_{i=1}^{D}$.
Under the condition of Theorem~\ref{theo cohen},
if  $q\geq 2$ then the scheme \eqref{method Cohen app}
is of order $1$ in the mean-square convergence sense.
\end{theorem}
\begin{proof}
We want to compare the scheme \eqref{method Cohen app} with the conservative method \eqref{method Cohen}.
The $k$th component of the one-step numerical scheme \eqref{method Cohen app} reads
\begin{align*}
\widehat{X}^{k}&=x^{k}+h\sum_{i=1}^{d}S^{ki}(\dfrac{x+\widehat{X}}{2})\sum_{\theta=1}^{D}b_{\theta}\partial_{i} I(x+c_{\theta}(\widehat{X}-x))\nonumber\\
&\quad+\sum_{r=1}^{m}\sum_{i=1}^{d}\Delta \hat{W}_{r}T_{r}^{ki}(\dfrac{x+\widehat{X}}{2})\sum_{\theta=1}^{D}b_{\theta}\partial_{i} I(x+c_{\theta}(\widehat{X}-x)).
\end{align*}
We next expand $S^{ki}(\dfrac{x+\widehat{X}}{2})$, $T_{r}^{ki}(\dfrac{x+\widehat{X}}{2})$ and
$\partial_{i}I(x+c_{\theta}(\widehat{X}-x))$ in Taylor series. For $\displaystyle\sum_{\theta=1}^{D}b_{\theta}=1$
and $\displaystyle\sum_{\theta=1}^{D}b_{\theta}c_{\theta}=\dfrac12$, we have
\begin{align*}
\widehat{X}^{k}
&=x^{k}+\sum_{i=1}^{d}S^{ki}\partial_{i}Ih+\sum_{r=1}^{m}\sum_{i=1}^{d}T_{r}^{ki}\partial_{i}I\Delta \hat{W}_{r}\\
&\quad+\dfrac12\sum_{r=1}^{m}\sum_{i,j=1}^{d}(\partial_{j}T_{r}^{ki}\partial_{i}I+T_{r}^{ki}\partial_{ij}I)\hat\Delta^{j}\Delta \hat{W}_{r}
+\hat{R}_{1},
\end{align*}
where
\begin{align*}
&\hat{R}_{1}=h\sum_{i=1}^{d}\Big(\dfrac{1}{2}\sum_{j=1}^{d}\partial_{j}S^{ki}(x)\hat\Delta^{j}+R_{S}\Big)(\sum_{\theta=1}^{D}b_{\theta}\partial_{i}I(x+c_{\theta}(\widehat{X}-x)))\\
&\quad+h\sum_{i=1}^{d}S^{ki}(x)\Big(\sum_{\theta=1}^{D}b_{\theta}\partial_{i}I(x+c_{\theta}(\hat{X}-x))-\partial_{i}I(x)\Big)\\
&\quad+\sum_{r=1}^{m}\sum_{i,j,l=1}^{d}\Delta \hat{W}_{r}\Big(T_{r}^{ki}(x)+\dfrac12 \sum_{j=1}^{d}\partial_{j}T_{r}^{ki}(x)\hat\Delta^{j}\Big)\dfrac12 \sum_{\theta=1}^{D}b_{\theta}c_{\theta}^{2}\partial_{ijl}I(x+\xi c_{\theta}(\widehat{X}-x))\hat\Delta^{j}\hat\Delta^{l}\\
&\quad+\sum_{r=1}^{m}\sum_{i,j,l=1}^{d}\dfrac14 \Delta \hat{W}_{r}\partial_{j}T_{r}^{ki}(x)\partial_{il}I(x)\hat\Delta^{j}
\hat\Delta^{l}
+h\sum_{r=1}^{m}\sum_{i=1}^{d} R_{S}
(\sum_{\theta=1}^{D}b_{\theta}\partial_{i}I(x+c_{\theta}(\widehat{X}-x))).
\end{align*}
Similar as in the proof of Theorem~\ref{theo cohen}, we define $\hat{R_{2}}$ as
\begin{align*}
\hat R_{2}=\dfrac12\sum_{r=1}^{m}\sum_{l,j=1}^{d}(\partial_{j}T_{r}^{kl}\partial_{l}I+T_{r}^{kl}\partial_{lj}I)
\Big(\hat\Delta^{j}-\sum_{i=1}^{m}\sum_{l=1}^{d}T_{i}^{jl}\partial_{l}I\Delta \hat{W}_{i}\Big)\Delta \hat{W}_{r}.
\end{align*}
It then follows that
\begin{align*}
\widehat{X}^{k}&=x^{k}+\sum_{i=1}^{d}S^{ki}\partial_{i}Ih+\sum_{r=1}^{m}\sum_{i=1}^{d}T^{ki}\partial_{i}I\Delta \hat{W}_{r}\nonumber\\
&\quad+\dfrac12 \sum_{r=1}^{m} \sum_{i,j=1}^{d}(\partial_{j}T_{r}^{ki}\partial_{i}I+T_{r}^{ki}\partial_{ij}I)(\sum_{l=1}^{d}T_{r}^{jl}\partial_{l}I)(\Delta \hat{W}_{r})^{2}\nonumber\\
&\quad+\sum_{i=1}^{m-1}\sum_{r=i+1}^{m}\sum_{l,j=1}^{d}(\partial_{j}T_{r}^{kl}\partial_{l}I+T_{r}^{kl}\partial_{lj}I)(\sum_{l=1}^{d}T_{i}^{jl}\partial_{l}I)
(\Delta \hat{W}_{i})(\Delta \hat{W}_{r})\nonumber\\
&\quad+\hat{R}_{1}+\hat R_{2},
\end{align*}
where $|E(\hat{R}_{1}+\hat R_{2})|=\mathcal{O}(h^{2})$, $(E|\hat{R}_{1}+\hat R_{2}|^{2})^{\frac12}=\mathcal{O}(h^{\frac32})$.
Comparing the scheme \eqref{method Cohen app} with the conservative method \eqref{method Cohen},
one concludes that the mean-square convergence order of the numerical approximation \eqref{method Cohen app} is $1$.
\end{proof}

The following result can be proved using similar techniques as in the proof of Theorem~\ref{theo_weak_1}.
\begin{theorem}\label{theo_weak_2}
Let $q$ be the order of the given quadrature formula $(c_{i},b_{i})_{i=1}^{D}$.
Assume that functions $S, T_r$ and $I$ satisfy the assumptions in Theorem~\ref{theo_weak_1}.
If $q\geq2$, then for all $n=0,1,\cdots, N$
and for small enough $h$, one has
\begin{equation*}
|E\phi(X(t_{n}))-E\phi(\widehat{X}_{n})|\leq Kh,\quad \text{ for }\quad\phi\in C_{P}^{4}(\R^{d},\R),
\end{equation*}
with a positive constant $K$ independent of $h$ and $n$.
I.\,e., the method \eqref{method Cohen app} has order of accuracy $1$ in the sense of weak approximations.
\end{theorem}

\section{Splitting approach}\label{sec-split}
Let us begin by recalling the SG formulation of our problem
\begin{equation}\label{splitting SG system}
\text dX(t)=S(X)\nabla I(X)\,\text dt+\sum_{r=1}^{m}T_{r}(X)\nabla I(X)\circ\,\text dW_{r}(t),
\end{equation}
where $S(X)$ and $T_r(X)$ are skew symmetric matrices.
The purpose of this section is to derive new numerical methods
for the above problem while preserving the conserved quantity $I$
on the basis of splitting techniques,
see also the works \cite{Zhang2011,Malham2008,m00} for similar ideas.

Let us first rewrite system \eqref{splitting SG system} as
\begin{equation*}
\text dX(t)=V_{0}(X)\,\text dt+\sum_{r=1}^{m}V_{r}(X)\circ\,\text dW_{r}(t),
\end{equation*}
where the vector fields $V_{0}$ and $V_{r}$ are defined by
\[V_{0}=\sum_{i=1}^{d}(S\nabla I)^{i}\partial_{i}\quad\text{ and }
\quad V_{r}=\sum_{i=1}^{d}(T_{r}\nabla I)^{i}\partial_{i},\quad r=1,\cdots,m.\]
Let $\Gamma$ be a set of multi-indices $\alpha$:
$\Gamma=\{\alpha=(\alpha_1,\alpha_2,\cdots,\alpha_\ell)\in\mathbb{N}_0^{\ell}\}$.
We denote by $|\Gamma|$ the number of elements of the set $\Gamma$.
We next split the above vector fields as
\[V_{0}=\sum_{\alpha\in\Gamma}V_{0}^{\alpha}\quad\text{ and }\quad V_{r}=\sum_{\alpha\in\Gamma}V_{r}^{\alpha},\quad r=1,\cdots,m,\]
such that there exist skew-symmetric matrices $S^{\alpha}$, $T_{r}^{\alpha}$ satisfying
$V_{0}^{\alpha}(X)=S^{\alpha}(X)\nabla I(X)$ and $V_{r}^{\alpha}(X)=T_{r}^{\alpha}(X)\nabla I(X)$ for $r=1,\cdots,m$.

The original system can then be divided into $|\Gamma|$ subsystems: $\forall \alpha\in\Gamma$
\begin{align}\label{subsystem}
\text  dX_{[\alpha]}&=V_{0}^{\alpha}(X_{[\alpha]})\,\text dt+\sum_{r=1}^{m}V_{r}^{\alpha}(X_{[\alpha]})\circ\text dW_{r}(t)\nonumber\\[-1.5ex]
\\[-1.5ex]
  &=(S^{\alpha}\nabla I)(X_{[\alpha]})\,\text dt+\sum_{r=1}^{m}(T_{r}^{\alpha}\nabla I)(X_{[\alpha]})\circ\text dW_{r}(t).\nonumber
\end{align}
It is thus natural to apply the conservative method \eqref{method Cohen} to each subsystems.
Denote by $\bar{X}_{[\alpha]}(\lambda, x):=\bar{X}_{[\alpha]}(\lambda)\circ\, x$, $\alpha\in\Gamma$,
$\lambda=1$ or $\frac12$ the corresponding one-step or half-step numerical approximation to \eqref{subsystem}.
We further define $\bar{Y}_{t,x}(t+h)$ by
\begin{align*}
\bar{Y}_{t,x}(t+h)=\bar{X}_{[\alpha_{1}]}(\dfrac12)\circ\, \bar{X}_{[\alpha_{2}]}(\dfrac12)\circ\,
\cdots \circ\, \bar{X}_{[\alpha_{|\Gamma|}]}(1)
\circ\,\cdots \circ\,\bar{X}_{[\alpha_{1}]}(\dfrac12)\circ\, x.
\end{align*}
Accordingly, using the above one-step numerical approximation,
we recurrently construct the composition scheme $\bar{Y}_{n}$, $n=0,1,\cdots,N$, by
\begin{equation}\label{composition}
\bar{Y}_{n+1}=\bar{Y}_{t_{n},\bar{Y}_{n}}(t_{n}+h),\quad \bar{Y}_{0}=X_{0}.
\end{equation}
Now, we introduce some notations and present a lemma, which lead to
the conclusion that the above composition scheme is of weak order $1$ and
of mean-square order $1$ in the case of commutative noise.
Denote $\phi_{[\alpha]}(\lambda,\tilde{x}):=\phi_{[\alpha]}(\lambda)\cdot \tilde{x}$, $\alpha\in\Gamma$,
$\lambda=1$ or $\frac12$, the numerical approximation defined by
\begin{equation*}
\phi_{[\alpha]}(\lambda,\tilde{x})=\text{exp}(\lambda h V_{0}^{\alpha}+
\lambda\sum_{r=1}^{m}\Delta\hat{W_{r}}V_{r}^{\alpha})\tilde{x},\quad \lambda=1 \text{ or }\frac12.
\end{equation*}
Accordingly, let $Z_{t,x}(t+h)$ be another one-step numerical approximation to the exact solution
of \eqref{splitting SG system} on $[t,t+h]$, which is defined by
\begin{align*}
Z_{t,x}(t+h)=\phi_{[\alpha_{1}]}(\dfrac12)\circ\,\phi_{[\alpha_{2}]}(\dfrac12)\circ\,
\cdots \circ\, \phi_{[\alpha_{|\Gamma|}]}(1)
\circ\,
 \cdots \circ\, \phi_{[\alpha_{1}]}(\dfrac12)\circ\, x.
\end{align*}
Using our previous results on mean-square and weak convergence orders, the following results can be proved
using similar ideas as in the proof of \cite[Lemma~3.2]{Zhang2011}.
\begin{lemma}\label{lemma1}
Assume that Milstein's scheme converges with mean-square order $1$ when applied to \eqref{subsystem}.
We have the following estimates for the one-step approximation $Z_{t,x}(t+h)$:
\begin{enumerate}
\item [{\rm(i)}] Under the condition of Theorem~\ref{theo cohen}, we have
\begin{align*}
&|E(X_{t,x}(t+h)-Z_{t,x}(t+h))|=\mathcal{O}(h^2),\nonumber\\
&(E|X_{t,x}(t+h)-Z_{t,x}(t+h)|^{2})^{\frac12}=\mathcal{O}(h^{\frac32}).
\end{align*}
\item [{\rm(ii)}] Under the condition of Theorem~\ref{theo_weak_1}, for $s=1,2,3,$ we have
\begin{align*}
|E\Big(\prod_{j=1}^{s}(X_{t,x}(t+h)-x)^{i_{j}}-\prod_{j=1}^{s}(Z_{t,x}(t+h)-x)^{i_{j}}\Big)|
=\mathcal{O}(h^2).
\end{align*}
\end{enumerate}
\end{lemma}

The above result permits us to show the next theorem.
\begin{theorem}\label{theo cohen3}
Assume that each subsystems \eqref{subsystem} have
commutative noise so that Milstein's scheme converges with mean-square order $1$.
The composition method \eqref{composition} has the following properties
\begin{enumerate}
\item [{\rm (i)}] It preserves exactly the scalar invariant $I$.
\item [{\rm (ii)}] Under the conditions of Theorem~\ref{theo cohen}, it has mean-square order of convergence $1$.
\item [{\rm (iii)}] Under the conditions of Theorem~\ref{theo_weak_1}, it is of weak order $1$.
\end{enumerate}
\end{theorem}
\begin{proof}
The first point is a direct consequence from the skew-symmetry of the matrices $S^{\alpha}$ and $T_r^{\alpha}$
and the result from Section~\ref{sec-prop}.

For the orders of convergence, we let $e_{1}=X_{t,x}(t+h)-Z_{t,x}(t+h)$
and $e_{2}=Z_{t,x}(t+h)-Y_{t,x}(t+h)$, then $e:=e_{1}+e_{2}=X_{t,x}(t+h)-Y_{t,x}(t+h)$
is the one-step approximation error of $Y_{t,x}(t+h)$.
Corresponding to the expressions of $Y_{t,x}(t+h)$ and $Z_{t,x}(t+h)$, we let
\begin{align*}
&x_{1}=x,\tilde{x}_{1}=x;\\
&x_{2}=\bar{X}_{[\alpha_{1}]}(\dfrac12)\circ\, x=\bar{X}_{[\alpha_{1}]}(\dfrac12,x_{1}), \quad \tilde{x}_{2}=\phi_{[\alpha_{1}]}(\dfrac12)\circ\, x=\phi_{[\alpha_{1}]}(\dfrac12,\tilde{x}_{1}),\\
&x_{3}=\bar{X}_{[\alpha_{2}]}(\dfrac12)\circ\,\bar{X}_{[\alpha_{1}]}(\dfrac12)\circ\, x=\bar{X}_{[\alpha_{2}]}(\dfrac12,x_{2}), \\ &\tilde{x}_{3}=\phi_{[\alpha_{2}]}(\dfrac12)\circ\,\phi_{[\alpha_{1}]}(\dfrac12)\circ\, x=\phi_{[\alpha_{2}]}(\dfrac12,\tilde{x}_{2}),\\
&...\\
&x_{|\Gamma|}=\bar{X}_{[\alpha_{1}]}(\dfrac12)\circ\, \bar{X}_{[\alpha_{2}]}(\dfrac12) \cdots   \bar{X}_{[\alpha_{|\Gamma|}]}(1)\cdots \bar{X}_{[\alpha_{1}]}(\dfrac12)\circ\, x=\bar{X}_{[\alpha_{1}]}(\dfrac12,x_{|\Gamma|-1}),\\
&\tilde{x}_{|\Gamma|}=\phi_{[\alpha_{1}]}(\dfrac12)\circ\, \phi_{[\alpha_{2}]}(\dfrac12) \cdots  \phi_{[\alpha_{|\Gamma|}]}(1)\cdots \phi_{[\alpha_{1}]}(\dfrac12)
\circ\,x=\phi_{[\alpha_{1}]}(\dfrac12,x_{|\Gamma|-1}),
\end{align*}
where $x_{|\Gamma|}=Y_{t,x}(t+h)$, $\tilde{x}_{|\Gamma|}=Z_{t,x}(t+h)$.
\begin{enumerate}
\item [(ii)] From Lemma~\ref{lemma1}, we know that $|Ee_{1}|=\mathcal{O}(h^{2})$ and
$(E|e_{1}|^{2})^{\frac12}=\mathcal{O}(h^{\frac32})$.
Next we estimate $e_{2}$ by induction on the index of the sequence $x_k-\tilde x_k$.
We recall that $\bar{X}_{[\alpha]}(\lambda,x)$ denotes the numerical solution to
the subsystem \eqref{subsystem} given by the scheme \eqref{method Cohen}.
From the mean-square convergence analysis in Theorem~\ref{theo cohen}
and comparing with Milstein's method, we know that
\begin{equation}\label{barX}
\bar{X}_{[\alpha]}(\lambda,x)=X^{mil}_{[\alpha]}(\lambda,x)+R_{[\alpha]}
\end{equation}
with $|ER_{[\alpha]}|=\mathcal{O}(h^{2})$ and $(E|R_{[\alpha]}|^{2})^{\frac12}=\mathcal{O}(h^{\frac32})$.
Here the expression of $X_{mil}^{[\alpha]}(\lambda,x)$ reads
\begin{align*}\label{milstein}
X^{mil}_{[\alpha]}(\lambda,x)&=x+\lambda h (S^{\alpha}\nabla I)(x)
+\sum_{r=1}^{m}\lambda \Delta W_{r}(T_{r}^{\alpha}\nabla I)(x)\nonumber\\
&\quad+\frac{\lambda^{2}}{2}\sum_{i=1}^{m}\sum_{r=1}^{m}\Lambda_{i}(T_{r}^{\alpha}\nabla I)(x)\Delta W_{i}\Delta W_{r}.
\end{align*}
On the other hand, from the definition of $\phi_{[\alpha]}(\lambda,x)$, it's not difficult to show that
\begin{equation}\label{phi}
\phi_{[\alpha]}(\lambda,\tilde{x})=X^{mil}_{[\alpha]}(\lambda,\tilde{x})+Q_{[\alpha]}
\end{equation}
with $|EQ_{[\alpha]}|=\mathcal{O}(h^{2})$ and $(E|Q_{[\alpha]}|^{2})^{\frac12}=\mathcal{O}(h^{\frac32})$.

We can now start the proof by induction. For the case $k=1$: Since $x_{1}=\tilde{x}_{1}$, one has
\[ |E(x_{2}-\tilde{x}_{2})|=\mathcal{O}(h^{2}),\quad (E|x_{2}-\tilde{x}_{2}|^{2})^{\frac12}=\mathcal{O}(h^{\frac32}).\]
Suppose now that $|E(x_{k}-\tilde{x}_{k})|=\mathcal{O}(h^{2})$ and $(E|x_{k}-\tilde{x}_{k}|^{2})^{\frac12}=\mathcal{O}(h^{\frac32})$.
The estimates
\[ |E(x_{k+1}-\tilde{x}_{k+1})|=\mathcal{O}(h^{2}),\quad (E|x_{k+1}-\tilde{x}_{k+1}|^{2})^{\frac12}=\mathcal{O}(h^{\frac32}),\]
follow from equations \eqref{barX}-\eqref{phi}.
This finally shows that $|Ee_{2}|=\mathcal{O}(h^{2})$ and $(E|e_{2}|^{2})^{\frac12}=\mathcal{O}(h^{\frac32})$.
The triangle inequality gives
\[|Ee|=\mathcal{O}(h^{2}),\quad (E|e|^{2})^{\frac12}=\mathcal{O}(h^{\frac32})\]
which shows that the composition method \eqref{composition} is of mean-square order $1$.
\item [(iii)] To prove the weak order of convergence of the composition method, we shall show that, for $s=1,2,3$,
\[ |E\Big(\prod_{j=1}^{s}(Y_{t,x}(t+h)-x)^{i_{j}}-\prod_{j=1}^{s}(Z_{t,x}(t+h)-x)^{i_{j}}\Big)|=\mathcal{O}(h^{2}).\]
This is again completed by induction.
The above estimates are satisfied for $k=1$.
Suppose now that the following estimates hold at the stage $k$,
\[ |E\Big(\prod_{j=1}^{s}(x_{k}-x)^{i_{j}}-\prod_{j=1}^{s}(\tilde{x}_{k}-x)^{i_{j}}\Big)|=\mathcal{O}(h^{2}),\quad s=1,2,3.\]
Next we show that they also hold at the stage $k+1$.
For ease of presentation, we only give details for the case $s=1$. The proofs for $s=2,3$ are similar.
From \eqref{barX} and \eqref{phi}, we have
\begin{align*}
(x_{k+1}-x)^{i_{1}}-(\tilde{x}_{k+1}-x)^{i_{1}}&=\Big((X_{[\alpha]}^{mil}(\lambda,x_{k})-x_{k})^{i_{1}}
-(X_{[\alpha]}^{mil}(\lambda,\tilde{x}_{k})-\tilde{x}_{k})^{i_{1}}\Big)\\
&\quad+\Big((x_{k}-x)^{i_{1}}-(\tilde{x}_{k}-x)^{i_{1}}\Big)+R_{[\alpha]}^{i_{1}}+Q_{[\alpha]}^{i_{1}}.
\end{align*}
Thus from the expression of $X_{[\alpha]}^{mil}$ and our assumptions, we obtain
\[\Big|E\Big((x_{k+1}-x)^{i_{1}}-(\tilde{x}_{k+1}-x)^{i_{1}}\Big)\Big|=\mathcal{O}(h^{2}).\]
A recurrence thus show the estimates, for $s=1,2,3$,
\[ |E\Big(\prod_{j=1}^{s}(Y_{t,x}(t+h)-x)^{i_{j}}-\prod_{j=1}^{s}(Z_{t,x}(t+h)-x)^{i_{j}}\Big)|=\mathcal{O}(h^{2}),\]
which, using  Lemma~\ref{lemma1}, show that the composition method \eqref{composition}
has weak order $1$ of convergence.
\end{enumerate}
\end{proof}

As before, one can show that if the numerical method \eqref{method Cohen app}
is used in the composition method, i.\,e. a quadrature formula of order $\geq2$ is employed, then
the mean-square as well as the weak order remain the same.

\section{Numerical experiments}\label{sec-num}
In this section, we present numerical experiments to support and supplement the above theoretical results.
\subsection{Experiment 1}
Let us first consider a problem satisfying the hypothesis of Theorems~\ref{theo cohen} and~\ref{theo_weak_1}:
a stochastic perturbation of a mathematical pendulum
\begin{align*}
\text d\begin{pmatrix}p\\q \end{pmatrix}=\begin{pmatrix}0 &-1\\1 & 0\end{pmatrix}\begin{pmatrix}p\\ \sin(q) \end{pmatrix}\text dt
+\begin{pmatrix}0& -\cos(q)\\ \cos(q) & 0\end{pmatrix}\begin{pmatrix}p\\ \sin(q) \end{pmatrix} \Big(c_1 \circ \text dW_{1}(t) +c_2\circ\text dW_{2}(t)\Big) 
\end{align*}
with initial values $p(0)=0.2$ and $q(0)=1$, $W_1(t)$ and $W_2(t)$ being two independent Wiener processes.
The energy $I(p,q)=\frac12 p^2-\cos(q)$ is an invariant of this problem.

\begin{figure}[htbp]
\centering
\subfigure{\includegraphics[width=0.4\textwidth]{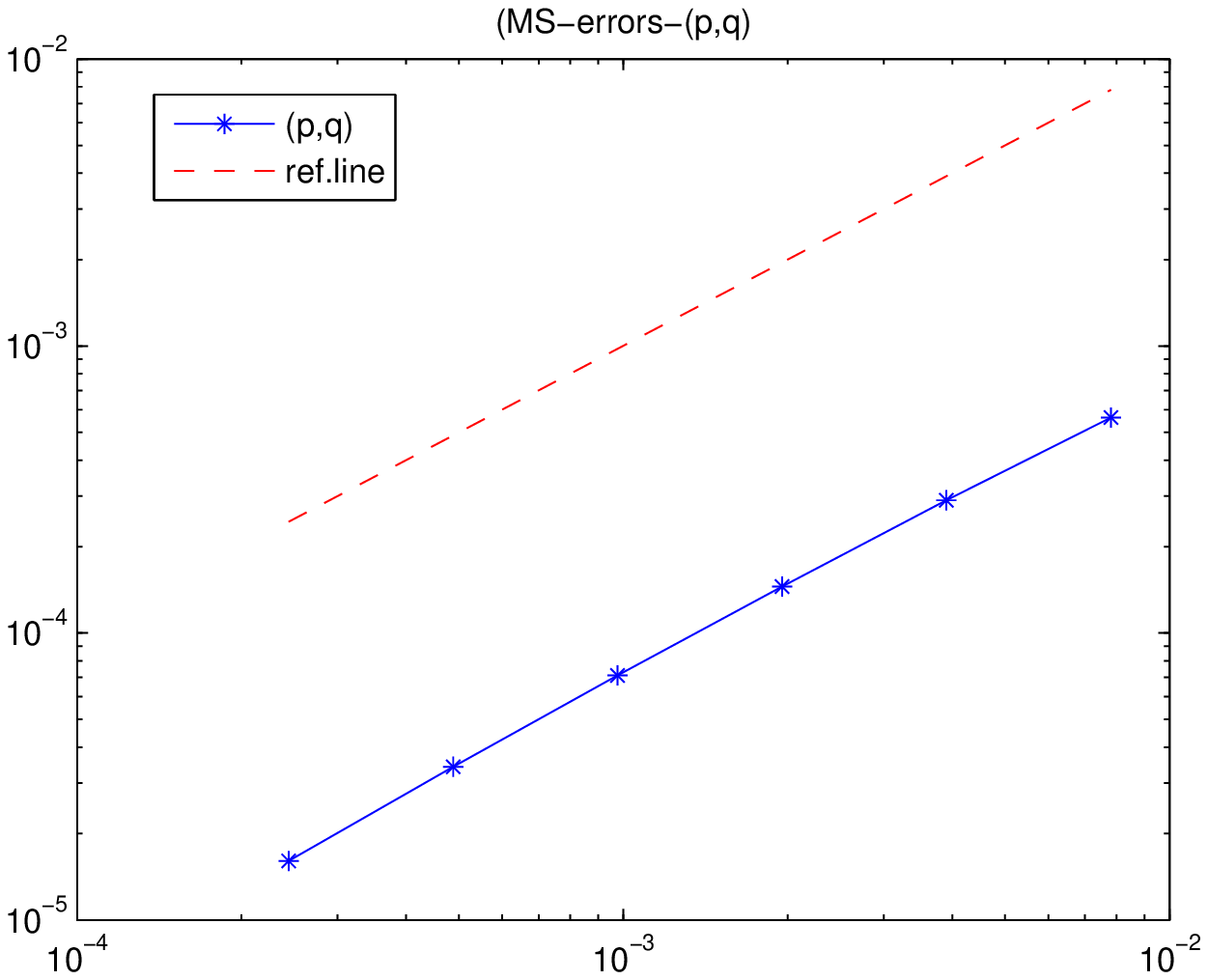}}
\subfigure{\includegraphics[width=0.4\textwidth]{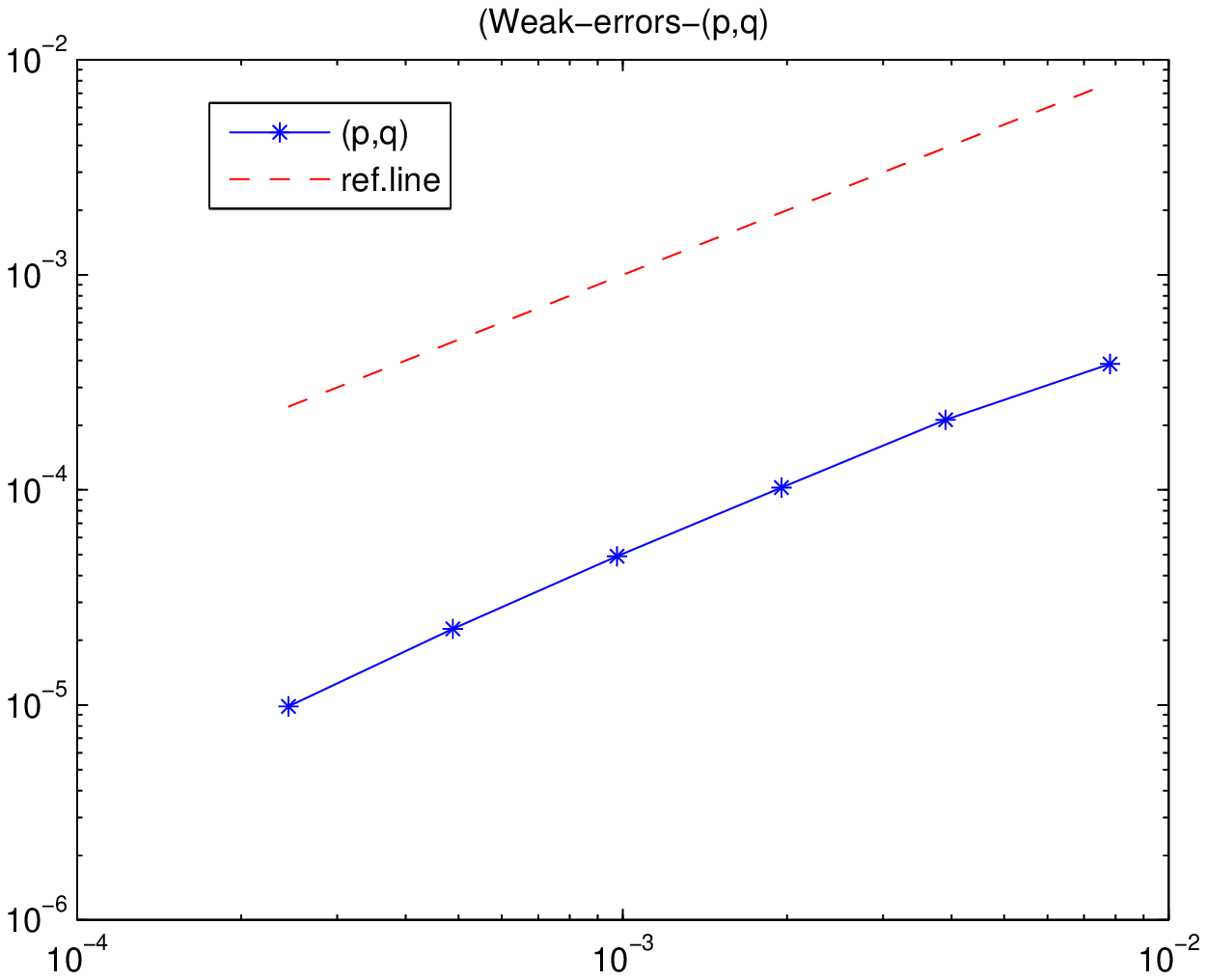}}
\caption{Stochastic pendulum with $c_1=1$ and $c_2=0.5$.
Left: Mean-square order, Right: Weak order.
Endpoint errors versus decreasing step sizes $h$
in log-log scale.
The reference lines have slope $1$. }
\label{E1figure1}
\end{figure}
Figure~\ref{E1figure1} displays the convergence order in both mean-square and weak sense.
From Theorems~\ref{theo cohen} and~\ref{theo_weak_1},
we know that the conservative scheme \eqref{method Cohen}
is of order $1$ in the mean-square, resp. weak sense for this stochastic mathematical pendulum problem.
The errors are computed at the endpoint $T_{N}=1$, the reference solution is computed using the step size
$h_{\text{exact}}=2^{-14}$ and the expectation is realised using the average
of $1000$ independent pathes. We can observe from Figure~\ref{E1figure1} (left) a mean-square order of
convergence one for the conservative scheme \eqref{method Cohen}.
The right picture shows the convergence order of
 $|E\big(\psi(p(T_N),q(T_N))-\psi(p_N,q_N)\big)|$ with the function $\psi(p,q)=\sin(p)+ q^2$.
The reference line has slope $1$, and we observe that the convergence orders
are consistent with our theoretical results.

\subsection{Experiment 2}
We are also interested in the following example, whose coefficients do not satisfy the hypotheses of
our main theorems. However, numerical results show that the convergence orders still coincide
with our theoretical assertions. We may say that our theory suits for a broader class of problems
than we claimed, and the study for the optimal assumptions is an open problem.
In order to illustrate this, we consider the cyclic Lotka-Volterra
(with commutative noise) \cite{Misawa1999}
\begin{align*}
\text d\begin{pmatrix}x^{1}\\x^{2}\\x^{3}\end{pmatrix}&=
  \begin{pmatrix}x^{1}(x^{3}-x^{2})\\x^{2}(x^{1}-x^{3})\\x^{3}(x^{2}-x^{1})\end{pmatrix}\text dt
  +\begin{pmatrix}x^{1}\\x^{2}\\-2x^{3}\end{pmatrix}\circ\text  dW_{1}
  +\begin{pmatrix}x^{1}\\-2x^{2}\\x^{3}\end{pmatrix}\circ\text  dW_{2}
+\begin{pmatrix}-2x^{1}\\x^{2}\\x^{3}\end{pmatrix}\circ\text  dW_{3}.
\end{align*}
This problem has the conserved quantity $I(x)=x^{1}x^{2}x^{3}$
and possesses the following skew gradient form \eqref{SG system}
\begin{align*}
\text d\begin{pmatrix}x^{1}\\x^{2}\\x^{3}\end{pmatrix}&=
  \begin{pmatrix}0 & 1 & -1\\ -1& 0 & 1\\ 1 & -1 & 0\end{pmatrix}\nabla I(x)\,\text dt+
  \begin{pmatrix}0 & \frac{1}{2x^{3}} & \frac{1}{2x^{2}}\\[2mm] -\frac{1}{2x^{3}}& 0 & \frac{3}{2x^{1}}\\[2mm]
 -\frac{1}{2x^{2}} & -\frac{3}{2x^{1}} & 0\end{pmatrix}\nabla I(x)\circ\text  dW_{1}\\[2mm]
  &+\begin{pmatrix}0 & \frac{1}{2x^{3}} & \frac{1}{2x^{2}}\\[2mm] -\frac{1}{2x^{3}}& 0 & -\frac{3}{2x^{1}}\\[2mm]
-\frac{1}{2x^{2}} & \frac{3}{2x^{1}} & 0\end{pmatrix}\nabla I(x)\circ\text  dW_{2}
  +\begin{pmatrix}0 & -\frac{1}{2x^{3}} & -\frac{1}{2x^{2}}\\[2mm] \frac{1}{2x^{3}}& 0 & \frac{3}{2x^{1}}\\[2mm]
\frac{1}{2x^{2}} & -\frac{3}{2x^{1}} & 0\end{pmatrix}\nabla I(x)\circ\text dW_{3}.
\end{align*}
We will now numerically integrate this problem on the interval $[0,1]$
using the initial values $x_{0}=(0.01,0.01,0.01)^{T}$.

From Theorems~\ref{theo cohen} and~\ref{theo_weak_1},
we know that the conservative scheme \eqref{method Cohen}
is of order $1$ in the mean-square, resp. weak sense.
Aiming at verifying these convergence orders,
we compute the errors at the endpoint $T_{N}=1$,
the expectation is realized using the average
of $1000$ independent pathes.
The left part of Figure~\ref{E2figure1}
displays the mean-square errors.
The lines with $*$ represent the relative errors
$\displaystyle\frac{(E|y(T_{N})-y_{N}|^2)^{1/2}}{(E|y(T_{N})|^{2})^{1/2}}$
with $y$ being $x^{1}$, $x^{2}$,
$x^{3}$ or $x$.
The right part of Figure~\ref{E2figure1}
displays the weak errors.
The lines with $*$ represents the relative errors
$\displaystyle\frac{|E(\psi(y(T_{N}))-\psi(y_{N}))|}{|E\psi(y(T_{N}))|}$
with the function $\psi(x)$
being $x^{1}x^{2}$, $x^{2}x^{3}$, $(x^{1})^{2}$ or $|x|^{2}$.
The reference solution $y(T_{N})$ is computed
using the stochastic midpoint scheme with stepsize $h=2^{-14}$
and the numerical solutions $y_{N}$ are computed using method \eqref{method Cohen}.
We observe the desired convergence orders for the conservative scheme \eqref{method Cohen}.

\begin{figure}[htbp]
\centering
\subfigure{\includegraphics[width=0.48\textwidth]{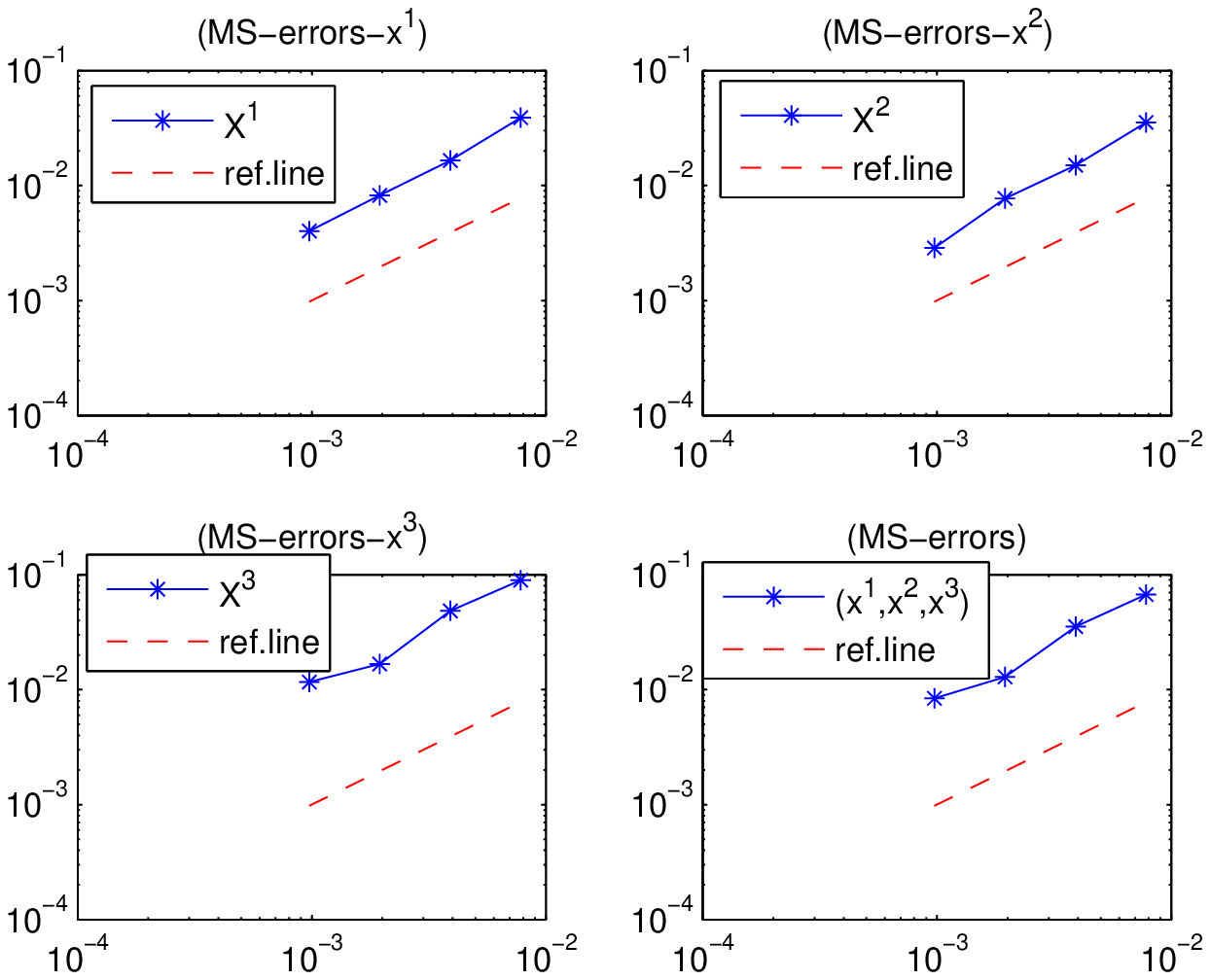}}
\subfigure{\includegraphics[width=0.48\textwidth]{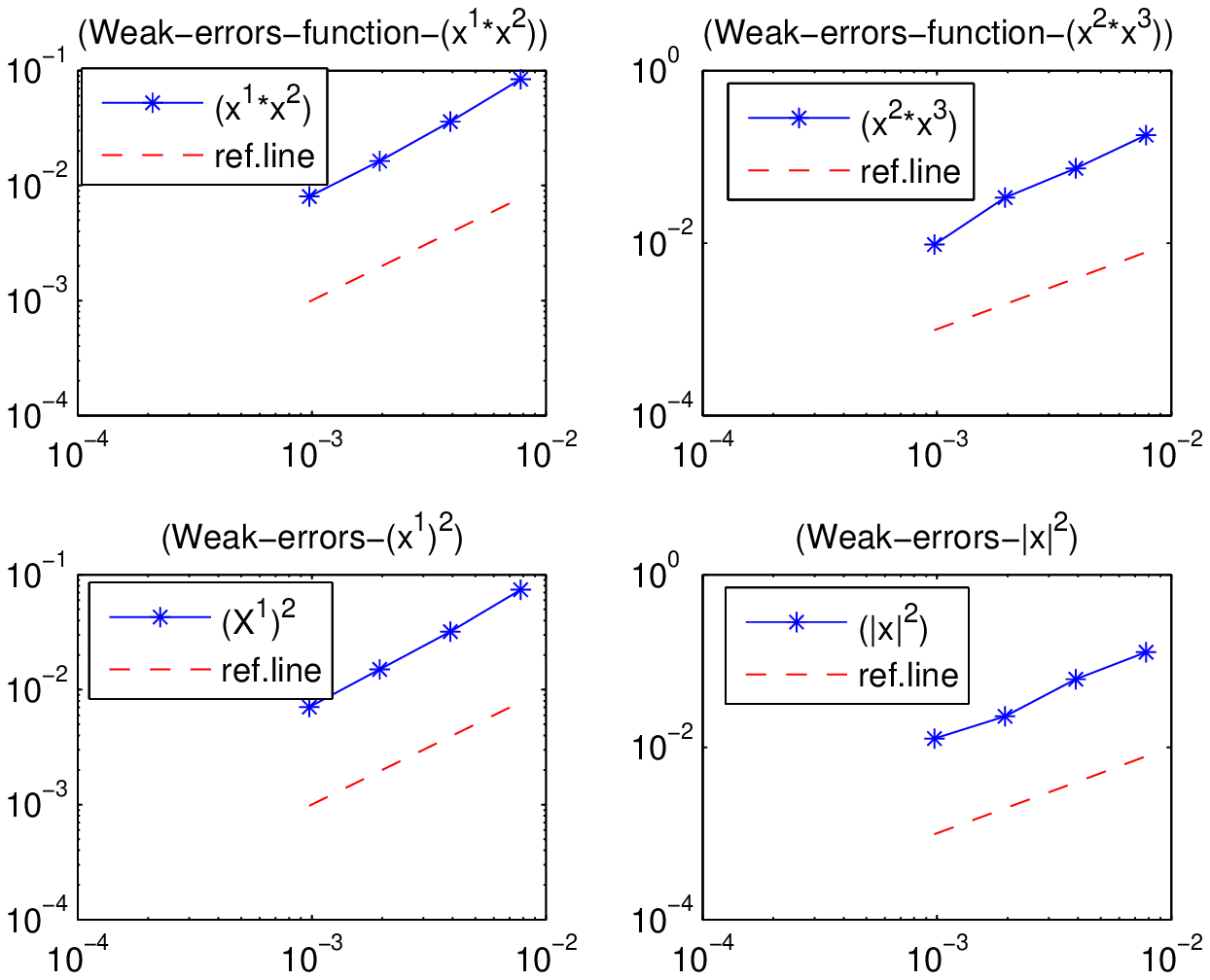}}
\caption{(Conservative scheme \eqref{method Cohen}.
Left: Mean-square order, Right: Weak order)
Endpoint errors versus decreasing step sizes $h$
in log-log scale for the stochastic
cyclic Lotka-Volterra system.
The reference lines have slope $1$.}
\label{E2figure1}
\end{figure}

We next repeat the same numerical experiments using the numerical method
\eqref{method Cohen app} with the classical midpoint rule.
We obtain similar plots as in the above experiments thus confirming
the convergence results from Theorems~\ref{theo cohen2} and~\ref{theo_weak_2}.
The plots are however not presented.

We finally apply a composition scheme to the cyclic
Lotka-Volterra system in order to verify the conclusions
of Theorem~\ref{theo cohen3}. To do this, we choose the set $\Gamma=\{12,13,23\}$
and consider $V_{0}^{ij}=S^{ij}\partial_{j}I\partial i-S^{ji}\partial_{i}I\partial_{j}$
and $V_{r}^{ij}=T_{r}^{ij}\partial_{j}I\partial i-T_{r}^{ji}\partial_{i}I\partial_{j}$
for $\alpha=ij\in\Gamma$. For the above systems,
the composition method \eqref{composition} reads
\[Y_{n+1}=\bar{X}_{[12]}(\dfrac12)\circ\,\bar{X}_{[13]}(\dfrac12)\circ\,\bar{X}_{[23]}(1)
\circ\,\bar{X}_{[13]}(\dfrac12)\circ\,\bar{X}_{[12]}(\dfrac12)\circ\, Y_{n}.\]

\begin{figure}[htbp]
\centering
\subfigure{\includegraphics[width=0.48\textwidth]{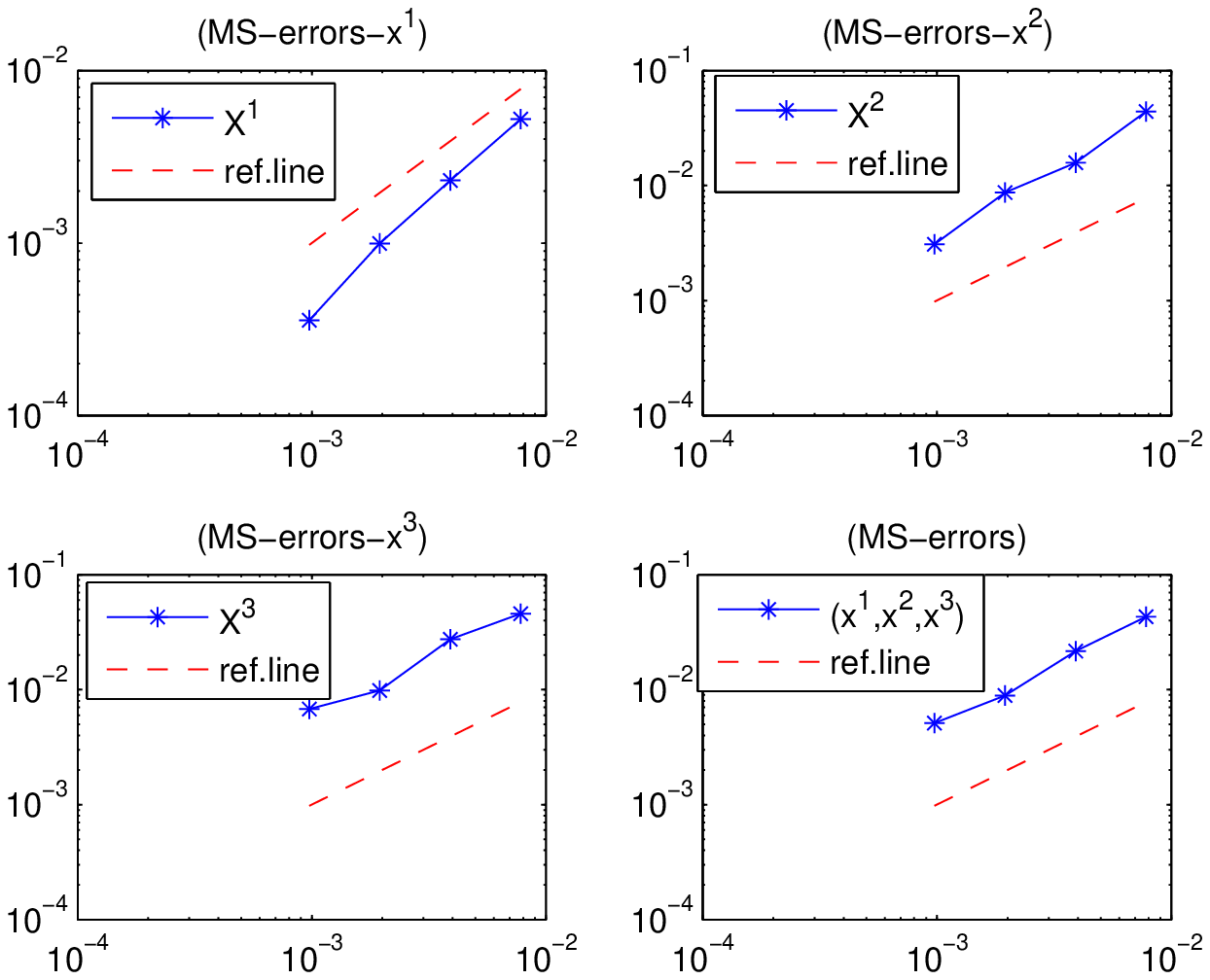}}
\subfigure{\includegraphics[width=0.48\textwidth]{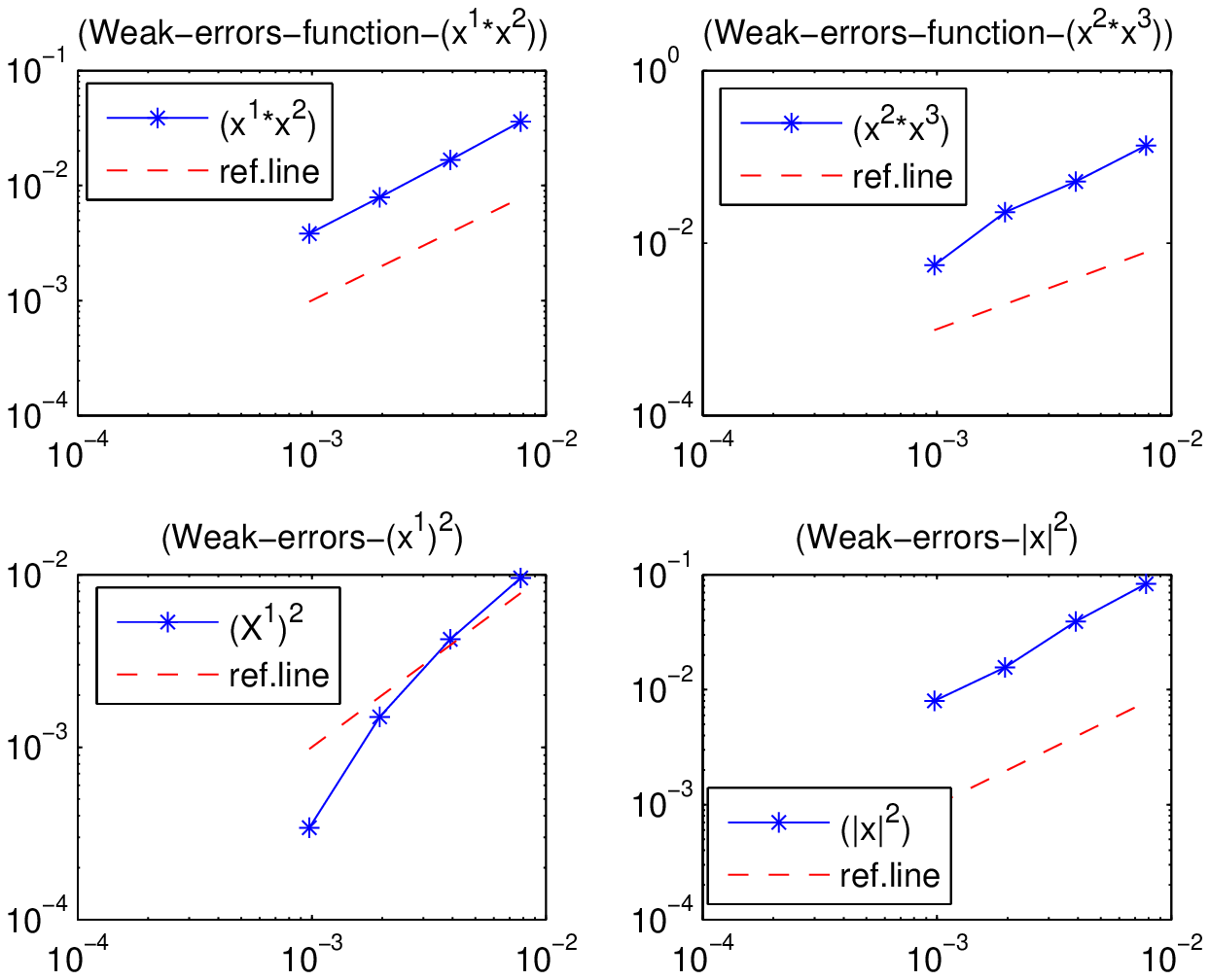}}
\caption{(Composition method \eqref{composition}.
Left: Mean-square order, Right: Weak order)
Endpoint errors versus decreasing step sizes
$h$ in log-log scale for the stochastic cyclic Lotka-Volterra system.
The reference lines have slopes $1$.}
  \label{E2figure3}
\end{figure}

The left part of Figure~\ref{E2figure3} presents
the mean-square errors.
The lines with $*$ represents the values of
$\displaystyle\frac{(E|y(T_{N})-y_{N}|^2)^{1/2}}{(E|y(T_{N})|^{2})^{1/2}}$ with $y$ being
$x^{1}$, $x^{2}$, $x^{3}$ or $x$.
The right part of Figure~\ref{E2figure3} presents the weak errors.
The lines with $*$ represents the values
of $\displaystyle\frac{|E(\psi(y(T_{N}))-\psi(y_{N}))|}{|E\psi(y(T_{N}))|}$
with the function $\psi(x)$
being $x^{1}x^{2}$, $x^{2}x^{3}$, $(x^{1})^{2}$ or $|x|^{2}$.
Again, the correct convergence orders are observed.

\section{Conclusion}
Based on the energy-preserving method for stochastic Poisson system \cite{Cohen2013}
and the equivalent skew gradient
system formulation of the original system \cite{Zhang2011}, we present a new invariant-preserving
method for general stochastic differential equations in the Stratonovich sense with
a conserved quantity.
%
We show that the invariant-preserving method converges with
accuracy order $1$ for commutative noise in mean-square sense.
In the commutative as well as non-commutative case, the weak convergence order of
the proposed method is $1$. Influences of the usage of a quadrature formula on the orders of convergence
are also investigated. Further, a conservative  composition method is studied:
mean-square convergence order $1$ for commutative noise and weak convergence order $1$ are obtained.
Finally, numerical experiments are presented to verify extend our theoretical results.
We will study multiple invariants-preserving methods for stochastic differential
equations in a future work.




\end{document}